\documentclass[reqno, 11pt]{amsart}

\usepackage{amsfonts, amsthm, amsmath, amssymb}
\usepackage{hyperref}
\hypersetup{colorlinks=false}

\usepackage{bbm}  
 
\usepackage{xcolor}

\usepackage[margin=1in]{geometry}

\usepackage{helvet}
\usepackage{xcolor}

\RequirePackage{mathrsfs} \let\mathcal\mathscr
   
\usepackage{hyperref} 
\usepackage{enumerate}

\usepackage{ dsfont }

\numberwithin{equation}{section}

\newtheorem{theorem}{Theorem}[section] 
\newtheorem{lemma}[theorem]{Lemma}
\newtheorem{proposition}[theorem]{Proposition}
\newtheorem{corollary}[theorem]{Corollary}

\theoremstyle{definition}

 \newtheorem*{acknowledgements}{Acknowledgements}
\newtheorem{remark}[theorem]{Remark}
\newtheorem{definition}[theorem]{Definition}
 \newtheorem*{notation}{Notation}

\renewcommand{\phi}{\varphi}

\renewcommand{\leq}{\leqslant}

\renewcommand{\geq}{\geqslant}
\renewcommand{\ge}{\geqslant}

\renewcommand{\c}{\mathbf{c}}

\renewcommand{\b}{\mathbf{b}}

\renewcommand{\r}{\mathbf{r}}

\renewcommand{\l}{\left}

\DeclareMathOperator{\Br}{Br}
\DeclareMathOperator{\Brsub}{Br_\textup{sub}}
\DeclareMathOperator{\Brvert}{Br_\textup{vert}}
\DeclareMathOperator{\rank}{rank}

\DeclareMathOperator{\Pic}{Pic}

\DeclareMathOperator{\moo}{mod} 
\renewcommand{\bmod}[1]{\,(\moo{#1})}

\let\emptyset\varnothing

\newcommand{\md}[1]{  \left(\textnormal{mod}\ #1\right)}
 
\renewcommand{\P}{\mathbb{P}}
\newcommand{\Q}{\mathbb{Q}}
\newcommand{\F}{\mathbb{F}}
\newcommand{\N}{\mathbb{N}}
\newcommand{\R}{\mathbb{R}}

\newcommand{\Z}{\mathbb{Z}}
 
\renewcommand{\r}{\right}
\renewcommand{\b}{\mathbf}
\renewcommand{\c}{\mathcal}
\renewcommand{\epsilon}{\varepsilon}

\renewcommand{\leq}{\leqslant}
\renewcommand{\geq}{\geqslant}
\renewcommand{\#}{\sharp}

\DeclareMathOperator*{\Osum}{\sum{}^*}

\newcommand{\beq}[2]
{
\begin{equation}
\label{#1}
{#2}
\end{equation}
}

\title 
[The number of soluble fibres 
on generalised Ch\^atelet varieties]
{Local solubility in  
generalised Ch\^atelet varieties}

\author{Kevin Destagnol}  
\address{Laboratoire de mathématiques d'Orsay \\ Université Paris Saclay \\   Orsay  \\  France}  
\email{kevin.destagnol@universite-paris-saclay.fr}

 \author{Julian Lyczak}  
\address{Department of Mathematics \\ 
 Universiteit Antwerpen
\\      Antwerpen  \\  Belgium}  
\email{julian.lyczak@uantwerpen.be}

\author{Efthymios Sofos} 
\address{Dipartimento di Matematica\\
Universit{\`a} di Roma
 Tor Vergata\\  Italy}
\email{efthymios.sofos@uniroma2.eu}

\subjclass[2020] {14G05, 
11N37, 
11P55. 
 }

\begin{document}

\vspace{-4cm}

\begin{abstract}We obtain asymptotic formulas for averages of general
multivariate  arithmetic functions evaluated at polynomial arguments 
using recent  work of Rydin Myerson 
and Rome--Yamagishi.  We give several applications of our results in analytic number theory and arithmetic geometry.
For example, we improve on the number of variables needed to prove the Hasse principle for certain polynomial systems, 
and we count the number of fibers with a rational point in families of high-dimensional Ch\^atelet varieties, allowing for arbitrarily large subordinate Brauer groups.
\end{abstract}  
\vspace{-4cm}
\maketitle
 
\vspace{-0,4cm}

\setcounter{tocdepth}{1}

\tableofcontents

\section{Introduction} \label{s:intro}
For many types of counting problems in arithmetic geometry there are 
conjectures for the order of growth and the leading constant, 
see for instance~\cite{LRS,LoSa,MR4307130,peyre,MR1679841}.  A common 
part of the leading constant is
the volume of an adelic subset, which is described by elements of some 
Brauer group. In the case of unramified Brauer group elements, the adelic 
subset is given by finitely many congruence conditions as the invariant maps are constant at all but finitely many places. 
Harari~\cite[Th\'eor\`eme~2.1.1]{MR1284820} proved 
 that, for ramified elements, the invariant maps are
 non-constant at a set of primes of positive density. 
 We will describe a counting problem where the adelic volume will involve these infinitely many congruence conditions.

\subsection*{Counting result}
We consider $R$ forms \( f_1, \dots, f_R \in \mathbb{Z}[x_0, \dots, x_n] \) of the same degree \( d \), generic as in Corollary~\ref{cor:birch},
with 
  \( dR\equiv 0 \md 2  \). For \( B \geq 1 \) and a square-free integer \( D \neq 1   \)
let 
\beq{def:countingfunction}{
N(B) := \#
\l\{\b t \in \P^n(\Z)\colon H(\b t) \leq B, \ \prod_{i=1}^R
f_i(\b t) 
\textrm{ is a norm from } \mathbb Q(\sqrt D)/\Q\r\}
,}  where $H$ is the  the anticanonical height function given by  $H(\b t):=\max\{|t_0|,|t_1|,\ldots, |t_n|\}^{n+1}$.
\begin{theorem}
\label{thm:mainthrm} Let $f_1,\ldots,f_R$ be as above
and fix any $A>0$. 
Then for $B\geq 3$ we have 
$$N( B)=\frac{\gamma}{2} \frac{(n+1)^{R/2}}{\zeta(n+1)}
 \l(\frac{2 }{\sqrt {\pi d} }\r)^R
 \frac{B }{(\log B)^{R/2}}
+O\l(\frac{B}{(\log B)^{R/2}}
\frac{1}{(\log \log \log B)^{A}}
\r),$$ 
 where $\zeta$ is the Riemann zeta function 
 and $\gamma$ is defined by 
$$ \gamma:= 
\sum_{\substack{ (s_1,\ldots, s_R) \in \{-1,1\}^R\\
D<0\Rightarrow s_1\cdots s_R >0 }} 
\gamma(\b s )  \cdot 
\mathrm{vol}
 \l\{\b t \in [-1,1]^{n+1}:   \mathrm{sign}(f_j(\b t ))=s_j \,\, \forall j\in \{1,\ldots,R\}\r\} 
.$$ To describe $\gamma(\b s )  $
we let  $\mu_p$ be the standard $p$-adic Haar measure on $\Z_p^{n+1}$
for a prime $p$
and   let $\prod_{p\leq T} \mu_p$
be the product measure on  
$\prod_{p\leq T} \Z_p^{n+1}$. Then $\gamma(\b s )$ is given by  
$$
 \lim_{T \to \infty}  
\l[
\prod_{\substack{p \ \mathrm{ prime} \\ p \leq T} }
\left(1-\frac1p\right)^{-\frac R2} 
\mu_p
\r]
\hspace{-0.1cm}
\l \{\b t \in  \l ( 
\prod_{p \leq T} \Z_p\r )^{n+1}\hspace{-0.3cm} \colon \ 
\hspace{-0.2cm}
\begin{array}{l} (D,  f_1(\b t ) \cdots f_R(\b t ) )_{\Q_p}=1 \ 
\forall p \leq T,  \\ \left(\frac D{s_if_i(\b t)/\prod_{p \leq T} 
p^{v_p(f_i(\b t))}}\right) = 1 \  \forall i\in \{1,\ldots,R\}\end{array}
\hspace{-0.1cm}
\right\}
,$$  where 
$(\cdot,\cdot)_{\Q_p}$ is the Hilbert symbol 
 in $\Q_p$ 
and $(\frac{D}{\cdot})$ is the   Kronecker symbol.
\end{theorem}

\begin{remark}[Leading constant]
We will see in~\eqref{eq:eulerproducts!}-\eqref
{eq:you are revolting, trully disgusting}
that $\gamma $ 
is a linear combination of $4^R$  Euler products and  
in Proposition~\ref{prop:comparing constants}
we shall prove that 
the leading constant  agrees 
with~\cite[Conjecture 3.8]{LRS}.
The constant $\gamma$ is not forced to be 
strictly positive, for example, by Corollary~\ref{cor:birch}.
\end{remark} \begin{remark}[Brauer group]
The varieties appearing in Theorem~\ref{thm:mainthrm} may possess 
subordinate Brauer groups of arbitrarily large size. 
This will be verified in 
Proposition~\ref{prop:subordinate brauer group}. 
\end{remark}  \begin{remark}[Reciprocity at infinitely many primes]
\label{rem:infreciproc}
The Kronecker symbol condition
 in the definition of $\gamma(\b s )$    gives a constraint at all 
 the  primes $p$ which are inert in $\Q(\sqrt D)$. 
  This behavior   is not typical    in the Tamagawa number 
in  Peyre’s constant for  Manin’s conjecture~\cite{peyre}, as 
that usually involves unramified Brauer group elements. 
In the context of counting locally soluble varieties 
Loughran was the first to study this phenomenon;
he introduced following Serre and fully developed the theory of the 
subordinate Brauer groups in~\cite{MR3852186}; in fact, 
~\cite[Theorem 2.13]{MR3852186} is a version of Harari's work
~\cite[Th\'eor\`eme~2.1.1]{MR1284820}
for the subordinate Brauer group. The present paper is inspired by~\cite{MR3852186} and 
proves an asymptotic for an  explicit variety with infinitely many 
reciprocity conditions.\end{remark}

\begin{remark}{(Generalised Ch\^atelet varieties)}
    Ch\^atelet surfaces are defined as smooth proper models of affine varieties of $\mathbb{A}_{\Q}^3$ of the shape 
    $$
    X^2-aZ^2=F(X,1)
    $$
    where $a$ is a non-square integer and where $F$ is a integral binary form of degree 4. Those surfaces, introduced by Ch\^atelet  \cite{Ch1,Ch2}, are arithmetically and geometrically very rich. The historical example given by Swinnerton-Dyer in 1971 of varieties for which both the Hasse principle and weak approximation fail are for instance Ch\^atelet surfaces. Ch\^atelet surfaces have been extensively studied both from the geometric aspect (see for instance \cite{CTS,CTSSD} where the authors show that the Brauer--Manin obstruction is the only one to the Hasse principle and to weak approximation) and more quantitatively. The study of Manin's conjecture, namely of the rational points of bounded height, for Châtelet surfaces started with the case $a=-1$ \cite{BBP,BB,BT,D} and culminated recently with the work of Woo \cite{Woo} settling all cases. All of these proofs make use of the natural conic fibration and it is therefore a natural question to ask how many fibers contribute to the final asymptotic. The varieties under consideration in Theorem~\ref{thm:mainthrm} are natural multivariate generalisations of Ch\^atelet surfaces.
\end{remark}

The analytic core of the proof of Theorem~\ref{thm:mainthrm} lies in the development of a new  
counting tool relying on Birch circle method, this explains the nature of the condition on the polynomials $f_i$ in the statement described in Corollary~\ref{cor:birch}. This tool will be described in the next subsection. Let us note that the counting problem still makes sense when $dR$ is odd. However, both the analytical techniques required to proving an asymptotic for the counting problem and the geometry of the interpretation change. For this reason we focus on the case where the total degree $dR$ is even.

\subsection{Averages of arithmetic functions over values of multivariable polynomials}

The analytic argument relies on developing a general tool 
that gives  asymptotics 
for $$  \sum_{\b t \in \Z^{n+1}\cap [-P,P]^{n+1} } 
k(f_1(\b t ), \ldots, f_R(\b t )  ) ,$$ 
where $f_1,\dots,f_R$ are generic polynomials as in 
Corollary~\ref{cor:birch},
provided only that 
$k:\N^R\to \mathbb C$
has an average
over arithmetic progressions of 
small modulus. The statement is in
Theorem~\ref{thm:vachms} 
and its proof occupies~\S\ref{s:circlemethdtul}. 
Due to its technical nature we avoid placing the statement 
here  and we instead give some of its applications.\\
\newline
Let $\mu$ be the M\"obius function.
\begin{theorem}[Chowla for Birch polynomials]
\label{ref:chowla} Let $f_1,\ldots, f_R \in 
\mathbb Z[x_0,\ldots,x_n]$ be polynomials as in    Corollary~\ref{cor:birch}. Then \[
\lim_{P\to \infty}  
\frac{1}{P^{n+1}}  
\sum_{\substack{\b t \in \Z^{n+1} \cap [-P,P]^{n+1} \\ \prod_i f_i(\b t) \neq 0 }}  
\mu(|f_1(\b t)|) \cdots \mu(|f_R(\b t)|) = 0.\]  
\end{theorem} This generalises~\cite[Theorem 1.6]{destagnolR=1}; we do not give details on the
proof since it follows by using 
Theorem~\ref{thm:vachms} 
instead of~\cite[Theorem 1.1]{destagnolR=1}
in~\cite[\S 3.1]{destagnolR=1}.

\begin{theorem}[Bateman--Horn 
for  multivariate polynomials]
\label{ref:batman} Let $f_1,\ldots, f_R \in 
\mathbb Z[x_0,\ldots,x_n]$ be polynomials as in Corollary~\ref{cor:birch}. Assume that 
$\mathcal B$ is a box  as in 
Definition~\ref{def:box}
 inside of which every $f_i$ stays strictly positive.
Then  \begin{align*}&
\lim_{B\to \infty}
\frac{\#\{\b t \in \Z^{n+1} \cap P\mathcal B: 
\textrm{ each } f_i(\b t ) \textrm{ is prime}\}
}{P^{n+1} (\log P)^{-R}} \\=&
\frac{ \mathrm{vol}(\mathcal B) }{d^R} 
\prod_{\substack{ p \  \mathrm{ prime } \\ p=2}}
^\infty 
  \frac{1-  
\#\{\b t \in \F_p^{n+1} : 
 f_1(\b t ) \cdots f_R (\b t ) =  0  \}p^{-n-1}}{(1-p^{-1})^R} 
.\end{align*}\end{theorem}
 This   result generalises \cite[Theorem 1.3]{prime} to multiple polynomials. Again we do not give details 
 for the proof as it is 
 similar to the one of 
 ~\cite[Theorem 1.4]{destagnolR=1}; one only needs
to feed the Siegel--Walfisz theorem for primes in arithmetic progressions 
into Theorem~\ref{thm:vachms}. \\
\newline
\indent
Recall 
the well-known 
Titchmarsh divisor problem which asks for asymptotics for the sum 
$$\sum_{\substack{ t \leq x \\ t \,\, {\rm{ prime}}} } 
\tau(t-1)
,$$ where $\tau$ is the divisor function. 
It was solved by Linnik, see 
\cite{Halberstam}
for a   proof via  the Bombieri--Vinogradov theorem
and the work of Assing--Blomer--Li \cite{1484.11194} 
for   recent developments.
Theorem~\ref{thm:vachms} will 
allow us in \S\ref{prooftitchm}
to prove 
 results for 
multiple arithmetic functions such as $k(m_1,m_2)=\mathds{1}_{\mathcal{P}}(m_1) \tau(m_2)$, where $\mathds{1}_{\mathcal{P}}$ is the indicator function of the primes.
For a prime $p$  and   integers $m\geq 0, \ell \geq 1$,
we let  $\tau_{p^m}(\ell) := 1 + \min\{v_p(\ell), m\}$. 
As $\tau(p^k) =   \lim_{m \to \infty} \tau_{p^m}(p^k)$,
the function $ \prod_{p|\ell}  \tau_{p^m} $ 
 models   $\tau $.

\begin{theorem}[Titchmarsh for  multivariate polynomials]
\label{ref:ritchmarsh} Let  $f_1, f_2 $ be integer polynomials
as in Corollary~\ref{cor:birch} with $R=2$.  Assume that 
$\mathcal B$ is a box  as in 
Definition~\ref{def:box}
 inside of which both $f_i$ stay strictly positive.
Then  \[ \lim_{P\to \infty}P^{-n-1} \! \!
\sum_{\substack{ \b t \in \Z^{n+1} \cap P\mathcal B\\
f_1(\b t)  \ \mathrm{ prime}  } } 
\tau(f_2(\b t  ) ) 
=\mathrm{vol}(\c B)
\prod_{\substack{p \ \mathrm{ prime} \\ p=2} }^\infty \left( 1-\frac{1}{p}\right)^{-1}
\lim_{m\to \infty}
p^{-m (n+1)} \! \!
\sum_{\substack{ \b x \in (\Z/p^m\Z)^{n+1} \\ 
p\nmid f_1(\b x ) } } \tau_{p^m}(f_2(\b x )) 
.\]\end{theorem} 
As another application we show that one needs a smaller number of variables  for the solubility in the case of hypersurfaces which are pullbacks. Let 
$F$ be an integer smooth form of degree $d_2$ 
in $R$ variables defining a projective variety $V$ and let 
$f_1,\ldots, f_R$ be integer forms 
of degree $d_1$ in $n+1$ variables so that the morphism 
$$ \Phi: \mathbb P^n  \dashrightarrow \mathbb P^{R-1}, 
\ \ \ \ 
[x_0 : \dots : x_n] \longmapsto [f_1(\mathbf{x}) : \dots : f_R(\mathbf{x})]$$
is well-defined. The hypersurface $\Phi^{-1}(V) \subseteq \mathbb P^n$ is described by $F(f_1(\b x), \ldots, f_R(\b x ) )=0 $ and has degree $d:= d_1 d_2$.
We shall prove a Hasse principle for  such varieties 
when  
the number of variables  $n+1$     exceeds 
a quantity that is exponentially smaller than 
the number of variables required for smooth forms of degree $d$  in the classical Birch work \cite{MR0150129}.
\begin{theorem}
\label{singular_hasse_principle} 
Let $n,d_1,d_2$ be strictly positive integers such that 
$R \geq 1 + (d_2-1)2^{d_2}$ and $$n\geq  d_12^{d_1} R + R
.$$
Let   $f_1,\ldots, f_R
\in \mathbb Z[x_0,\ldots, x_n]$ be forms of 
degree $d_1$     as in Corollary~\ref{cor:birch}.
Assume that 
$F\in \mathbb Z[x_1,\ldots, x_R]$ is a smooth 
form of degree $d_2$.
Then the hypersurface 
$$ F( f_1(\b x ), \ldots, f_R( \b x ) ) =0$$
satisfies the Hasse principle. Furthermore, there exists $\eta=\eta(F,\b f)>0$ such that$$
\#\left\{\b  t \in \mathbb Z^{n+1}\cap [-P,P]^{n+1}: 
F(\b f (\b t ) )=0\right\}  \gg  \frac{
 P^{n+1-d_1d_2} }{(\log P)^{\eta}}
  $$ for all $P\gg_{F,\b f} 1$. \end{theorem} 
  \begin{remark}[Exponentially smaller
  number of variables]
  Let 
$$n_0 :=d_12^{d_1} (1 + (d_2-1)2^{d_2}) + 1 + (d_2-1)2^{d_2}
\ \ \ \textrm{ and } \ \ \ 
n_1=1+(d_1d_2-1) 2^{d_1d_2}$$
so that $n_0$ is the   number of variables for 
$F( \b f (\b x ))=0$ to satisfy the Hasse principle by Theorem~\ref{singular_hasse_principle} 
and $n_1$ is   
the number of variables required for a smooth form  of the same degree as $F(\b f (\b x ))$ to satisfy the Hasse principle 
in the  work of Birch  \cite{MR0150129}. One can check that 
$n_0<n_1$ whenever both $d_i$   strictly exceed  $1$. For example, when $d_1=3=d_2$ then $n_0=425 $ and $n_1=4097$.
Furthermore, $n_1$ is substantially larger than $n_0$,
for example, when $d_1=d_2=d\to \infty$
then 
$n_0 \sim d^2 2^{2d}$  and $
n_1\sim d^2 2^{d^2}$.
\end{remark}

   Theorem 
\ref{singular_hasse_principle}  is proved in \S \ref{No, non vedete mai}
by counting   points of bounded height  
by  combining 
 Theorem~\ref{thm:vachms} 
 with the   work of 
El-Baz--Loughran--Sofos \cite{elbaz} 
on equidistribution of integer solutions on 
hypersurfaces  of large dimension.

\subsection{Multiplicative models}
In the proof of Theorem~\ref{thm:vachms}, we opted for a ``multiplicative model" approach rather than the 
more standard method based on exponential sums, as exemplified by~$\mathfrak{S}^\flat$ in Lemma~\ref{lem:cramer}. This choice ensures that the leading constant in Theorem~\ref{thm:vachms} is particularly convenient for the counting applications mentioned above. 
Notably, in the proof of Theorem~\ref{thm:mainthrm}, this constant will   take  ftheorm of a   sum of $4^R$
Euler products. It is worth emphasizing that in applications of the circle method, it is 
somewhat uncommon for the leading constant to be
a sum of more than a single 
Euler product of \( p \)-adic densities but the formulation of Theorem~\ref{thm:vachms} 
covers such cases as well.

 \subsection{Proving Theorem~\ref{thm:mainthrm} via a Splitting Trick}  

We derive Theorem~\ref{thm:mainthrm} from Theorem~\ref{thm:vachms} in \S\ref{s:proofproofproof}. The core idea is a \textit{splitting trick}, which proceeds in two steps.  

In the first step, we show that for almost all~$\b{t}$ in~$N(B)$, any two of the values \( f_i(\b{t}) \) 
only
share  prime factors of nearly bounded size. This step is motivated by the codimension~$2$ Ekedahl sieve philosophy. However, unlike traditional applications of the geometric sieve, our setting involves a zero-density subset of~$\mathbb{P}^n(\mathbb{Z})$, dictated by the norm condition in~$\mathcal{N}_D$. This approach is developed in Lemmas~\ref{lem:codimY}--\ref{lem:codijhsdcontentomY}, where we apply a general circle method upper bound (Lemma~\ref{corl:uperbnd}) to reduce the problem to averages over~$\mathbb{Z}^R$. These averages are then analyzed in Lemmas~\ref{lem:shiuconseqnjdf8d8}-\ref{lem:shiucuwer6} using the large sieve.  

This first step ensures that we need only feed into the circle
method 
an equidistribution result for tuples \((n_1, \dots, n_R)
\in \mathbb{Z}^R\), where 
\(n_1\cdots n_R\) is a norm in
$\mathbb{Q}(\sqrt{D})/\mathbb{Q}$ and, crucially, 
the~\(n_i\) are nearly coprime in pairs.
Without the near coprimality, a Perron formula argument in~$\mathbb{C}^R$ would be required. 
The equidistribution result forms the second step of the splitting trick and is stated in Lemma~\ref{lem:adrianosyria}, with its proof occupying \S\ref{RV108}.  

Finally, in \S\ref{s:Julian the Apostate}, we demonstrate that the factor~\(2^R\gamma\) coincides with the product of the order of the subordinate Brauer group and the Tamagawa number, as predicted by~\cite[Conjecture 3.8]{LRS}.  

 \subsection{Previous results}  \label{s:bibliogggg}
Asymptotic estimates for~\eqref{def:countingfunction} in the case of small \( n \) have been previously studied. When \( n = 1 \), the more general conic bundle equation  
\[
\sum_{1\leq i<j\leq 3} f_{ij}(\mathbf{t}) x_i x_j=0
\]
was analyzed in~\cite{MR3534973}, where matching upper and lower bounds were established under the condition that the discriminant of the conic bundle has total degree at most 3. Additionally, Friedlander and Iwaniec~\cite{MR476673} proved matching upper and lower bounds for the density of integers \( t \) for which an irreducible quadratic polynomial \( f(t) \) represents a norm in a quadratic field. A similar problem for an irreducible cubic polynomial \( f \) was investigated by Iwaniec and Munshi~\cite{MR2580453}.  
More generally, no asymptotic results for~\eqref{def:countingfunction} are known 
in the literature when \( n = 1 \) and \( \deg(\prod_i f_i) \geq 2 \).
A very special case of the broader work of 
Loughran--Matthiesen~\cite[Theorem 1.1, Example 1.3]{MR4780494}
handles succesfully all $f_i$ that factor  as a product of linear 
polynomials over $\Q$.  

\begin{acknowledgements} 
Parts of this investigation
took
place when JL and ES
 visited the university Paris-Saclay during February 2025, the
generous hospitality and support of which is greatly appreciated. 
We also thank 
the Max Planck Institute 
for Mathematics in Bonn, 
where this collaboration 
started in 2018. 
ES
gratefully acknowledges
Régis de la Bretèche
for financial support that made 
it possible to visit Paris and collaborate on this work.
We are grateful to Daniel Loughran for his helpful comments concerning Remark~\ref{rem:infreciproc} and \S\ref{s:bibliogggg}. We also thank Jean-Louis Colliot-Thélène for numerous valuable suggestions on the final draft of this paper.  
Finally, we would like to thank
Simon L. Rydin Myerson for  
his help on Corollary~\ref{cor:birch} 
and  Damaris Schindler for valuable discussions 
that helped with the proof of Theorem \ref{singular_hasse_principle}.
\end{acknowledgements}

\begin{notation}
Throughout the paper $D$ will denote 
a square-free integer different from $1$.
Further, 
\beq{def:norms}{
\c N_D:=
\l\{ m \in \Z: \exists \alpha \in \mathbb Q(\sqrt D)
\textrm{ such that } 
m=\textrm{N}_{\mathbb Q(\sqrt D)/\mathbb Q}(\alpha)
\r\} .} The notation $f_i$ will always refer 
to $R$ integer forms  
in $n+1$ variables  and of 
the same degree~$d$.
They will     satisfy the assumptions of 
Corollary~\ref{cor:birch}.
For $v\in \Omega_{\mathbb{Q}}$
and $a,b \in \Q_v$
 we denote the Hilbert symbol  
of $x_0^2-a x_1^2=b x_2^2$ over $\Q_v$ 
by  $(a,b)_{\mathbb{Q}_v}$.  \end{notation}

\section{The circle method tool}
\label{s:circlemethdtul}

In \S\ref{s:valueset} we state the main theorem of this section, namely Theorem~\ref{thm:vachms}.
It regards averages of arithmetic functions in several 
variables over values of   polynomials.
We  provide some   applications other than 
Theorem~\ref{thm:mainthrm}
and remarks on how to use it.
In Lemma \ref{corl:uperbnd}   we  
prove an upper bound 
for these averages under 
no assumptions on the arithmetic function.
Finally, in~\S\ref{s:pprrooff}
we prove  Theorem~\ref{thm:vachms}.

\subsection{Averages of arithmetic functions over values of multivariable polynomials}   \label{s:valueset}  
 Theorem~\ref{thm:vachms},
   states  that if 
   $k:\N^R\to \mathbb C$ is 
   equidistributed in progressions
    then for ``well-behaved" 
polynomial systems  $(f_1,\ldots, f_R)
\in \Z[x_0,\ldots,x_n]^R$,
 the average 
 $$  \sum_{\b t \in \Z^{n+1}\cap [-P,P]^{n+1}} 
k(f_1(\b t ), \ldots, f_R(\b t )  ) $$  
can be estimated asymptotically. It is important
to note that $k$ need not be 
non-negative  nor multiplicative.  
In fact, our averages will be taken over   
 expanding  boxes
 $P\c B$, where $\c B$ is as follows: 
 \begin{definition}\label{def:box}
Let $\mathscr{B}\subset [-1,1]^{n+1}$ have 
  sides of length at most $1$ which are parallel 
to the coordinate axes. Denote
 $ b_i=2 \max_i |f_i(\c B)| $ and  
$ b=\max b_i$.
\end{definition}

For integer polynomials 
$f_1,\ldots,f_R$ in $n+1$ variables  the usual
singular series and singular integral are respectively 
denoted by $\mathfrak{S}$  and $\mathfrak{J}$
and are defined, provided convergence, as  
 \begin{align*}
\mathfrak{S} 
:&= \sum_{q=1}^{\infty} q^{-n-1} 
\sum_{\substack{a_1, \dots, a_R = 1 \\ 
\gcd(a_1, \dots, a_R, q) = 1}}^{q} 
\mathrm e\Big( -\frac{1}{q}\sum_{i=1}^R a_i  \nu_i \Big) 
\sum_{\b x \in (\Z/q\Z)^{n+1}} 
\mathrm e\Big( \frac{1}{q}\sum_{i=1}^R a_i f_i(\b x)\Big)  \\
  &=
  \prod_{\substack{ p \textrm{ prime}\\p=2}}^\infty
  \lim_{m\to \infty} 
\frac{\#\{\b t \in (\Z/p^{m}\Z )^{n+1}: 
\b f (\b t ) = \boldsymbol \nu \}}{p^{m(n+1-R) }}  
\end{align*}  and 
$$ \mathfrak{J}:= \int_{\boldsymbol \gamma \in \mathbb R^R}
\mathrm e\Big(  -\sum_{i=1}^R \gamma_i \nu_i P^{-d} \Big)
\int_{\boldsymbol \zeta \in  \mathscr{B} } 
\mathrm e\Big( \sum_{i=1}^R  \gamma_i f_i^\natural(\boldsymbol \zeta) \Big)
 \mathrm d \boldsymbol \zeta  
\mathrm d \boldsymbol \gamma ,$$ where $\mathrm e(z)=\mathrm e^{2 i \pi z}$ and $f^\flat$ denotes the top degree homogeneous part of a polynomial $f$.
Our analytic tool
will apply to the following systems of polynomials.

\begin{definition}[strong 
Hardy--Littlewood system]\label{stronghardylittlew}
 A set of integer polynomials $\{f_1,\ldots,f_R\}$ 
 in $n+1$ variables is said to be a strong 
Hardy--Littlewood system when the following three properties hold:
\begin{itemize}
\item there exists 
$\eta=\eta(\b f)>0$   such that for all $\b a $ and 
$q$ appearing in $\mathfrak S$ one has 
\begin{equation}\label{eq:singseries}
\left| \sum_{\b x \in (\Z/q\Z)^{n+1}} 
\mathrm e\left( \frac{1}{q}\sum_{i=1}^R a_i f_i(\b x) \right)
\right| \ll q^{n-R-\eta},
\end{equation}  where  the  implied constant 
depends only on $\b f$;
\item  for all       $\mathscr{B}$ 
 in Definition \ref{def:box} 
 and  all $\boldsymbol \gamma \in \mathbb R^R$ 
one has
\begin{equation}\label{eq:singint} \left| \int_{\boldsymbol \zeta \in  \mathscr{B} } 
\mathrm e\Big( \sum_{i=1}^R  \gamma_i f_i^\natural(\boldsymbol \zeta) \Big)
 \mathrm d \boldsymbol \zeta  \right| \ll (1+
 \max_i | \gamma_i|)^{-R-\eta},
\end{equation} 
where the  implied constant   depends only on $\b f $ and $\c B$, and
\item 
there is $\delta=\delta(\b f )>0$  
such that for all       $\mathscr{B}$ 
 in Definition \ref{def:box}  and all 
 $P \ge 1$,
 $\boldsymbol \nu \in \mathbb{Z}^R$ we have  
\begin{equation}\label{eq_simon_asymptotic}
 \#\{\mathbf{x} \in \mathbb{Z}^{n+1}\cap  P \mathscr{B}: \mathbf{f}(\mathbf{x}) = \boldsymbol \nu\}= \mathfrak{J} \mathfrak{S}
 P^{n+1-dR} + O(P^{n+1-dR-\delta}),
\end{equation} where  the  implied constant 
depends only on $\b f$
and $\mathscr{B}$ and we have indeed convergence of the singular series and integral.
\end{itemize}
\end{definition}

Next, we clarify the assumptions
on the arithmetic function $k$. 
Let~$R\in \N$, 
$\omega:[1,\infty)^R\to \mathbb C$ be of class $\c C^1$  and 
for any  $q\in \N$  and $\b a\in (\Z/q\Z)^R$ let  
$\rho(\b a,q)$ be any   number in $ \mathbb C$. 
Assume we are given a function  $k:\N^R\to \mathbb C$  
and define 
\begin{equation}
\label{eq:basicproperty}
E_{k,\omega,\rho}(x;q):= 
\sup_{x_1,\ldots,  x_R \in \R\cap  [1,  x]}
\max_{\substack{ \b a\in (\Z/q\Z)^R  \\ \gcd(\b a , q )=1 }}
\left| \sum_{\substack{ \b m  \in \prod_{i=1}^R (\N\cap [1,x_i])   
\\ \b m \equiv \b a \md q  }} 
k(\b m) -\rho(\b a,q)   \int_{   \prod_{i=1}^R [1,x_i ]  }  \omega(\b t) \mathrm d \b t  \right| .
\end{equation}
We shall   denote $E_{k,\omega,\rho}(x;q)$ 
by $E(x;q)$ to simplify   notation. 
We say that   $k$ is equidistributed in progressions modulo $q$ 
when 
$E(x;q)$   grows slower than the main term
$\int_{[1,x]^R} \omega(\b t ) \mathrm d \b t $.

We will only need to know equidistribution for progressions modulo
 a single integer $W_z$ that we now proceed to define.
 
\begin{definition}\label{def:wzez} For any $z\geq 2$ assume we have a   function  
$m_p(z):\{\textrm{primes}\}\to \Z\cap [1,\infty)$ and let  $$W_z:=\prod_{p\leq z} p^{m_p(z)} .$$ 
We shall assume that  $ m_p(z)$ is suitably large so that $\widetilde{\epsilon}(z)\to 0 $ as $z\to \infty$, where
  $$\widetilde{\epsilon}(z):=\sum_{p\leq z} \frac{1}{p^{1+m_p(z)}} .$$ \end{definition} 
  \begin{remark} \label{rem:billakos evans} The number   $\widetilde{\epsilon}(z)$ is an upper bound for the probability that a random integer is divisible by 
  the power   $p^{1+m_p(z)}$ for at least one prime $p\leq z $. For the remaining of the paper we   abbreviate $m_p(z)$ by $m_p$. For certain functions $k$
it is much more straightforward
to estimate    $E(x;W_z)$ for
square-full $W_z$, 
which explains why we allowed complete freedom in choosing 
the exponents $m_p(z)$.

Let us now  bound    
$\widetilde{\epsilon}(z)$.
Note that for $m\geq 2 $ we have 
$$\sum_{p\leq z} p^{-m}
\leq 
2^{-m}+
\int_2^\infty \frac{\mathrm dt }{t^m} =
2^{-m}+
\frac{ 2^{1-m} }{ m -1 } 
\leq \frac{3}{2^{m} }.
 $$
Hence, if $  m_p\geq 2$ for all $p$, then 
\beq{eq:adriano in siris by pergolesi}  
{\widetilde{\epsilon}(z) \leq  \frac{3}{2^{\min_p m_p}}.
}\end{remark}

\begin{theorem}\label{thm:vachms} Fix any $d,R,n \in \N$ and assume that 
 $f_1,\ldots,f_R \in \Z[x_0,\ldots, x_n]$ are polynomials 
as in Definition~\ref{stronghardylittlew}.
Let $s_1,\ldots, s_R \in \{-1,1\}$.  Let  $\c B\subset \R^{n+1}$ be as in Definition~\ref{def:box}.
For  any  $z\geq 2 $ let $W_z$ be  as in Definition~\ref{def:wzez}.
Assume that for all $q\in \N, \b a \in (\Z/q\Z)^R$ we are given 
complex numbers  $\rho(\b a , q )$ 
of modulus at most $q^C$ for some positive $C$ independent of $\b a $ and $q$. Then for any function
 $k:\N^R\to \mathbb C$  and  all   
 $P \geq 1$  we have   \begin{align*}
\frac{1}{P^{n+1}} 
&\sum_{\substack{ \b t \in \Z^{n+1}\cap P\c B \\ \min_j s_j f_j(\b t ) > 0 }}  k(s_1 f_1(\b t), \ldots, s_R f_R(\b t) ) 
\\
&=   
\left(\,\,   \int\limits_{\substack{ \b t \in \c B \\ 
 s_j f_j(\b t ) > P^{-d} \ \forall j} }  
\omega\left( P^d (s_j  f_j(\b t ))\right)  \mathrm d \b t 
\right)
 \sum_{\b t \in (\Z/W_z\Z )^{n+1} } \frac{ \rho((s_j  f_j(\b t ) ) ,W_z) }{ W_z^{{n+1}-R} }
\\
&+O\left( \frac{\|k \|_1}{P^{Rd}  }   (P^{-c}+\widetilde{\epsilon}(z)+z^{-c}) +\frac{E(bP^d; W_z ) W_z^R  }{P^{Rd}   } \right),
\end{align*} where the implied constant depends at most on $\b f$
and $\mathscr{B}$. Here $c=c(\b f )>0$ is an explicit positive constant. The quantities 
$E(\cdot)$, $\widetilde{\epsilon}(\cdot)$ and $b$
are respectively  given  in~\eqref{eq:basicproperty}
and Definition~\ref{def:box}
and $$
\|k \|_1:= \sum_{\boldsymbol \nu \in \N^R \cap [1 , b P^d]^R} |k(\boldsymbol \nu)| 
.$$ 
\end{theorem} 
 We next give some 
 special cases 
 of   polynomial systems 
to which Theorem~\ref{thm:vachms} applies.  
\subsection{Polynomials of general shape}
\begin{corollary}\label{cor:birch}
Assume that $d,n,R$ are strictly positive  integers satisfying 
$$  n+1 > d2^d R + R.
$$We assume that $f_1,\ldots,f_R\in \mathbb Z[x_0,\ldots,x_n]$ 
are such that each $f_i^\natural$   has degree $d$
and that the Jacobian matrix
$(  \partial f_i^\natural(\mathbf{x})/\partial x_j  )_{i,j}$ has full rank $R$
at every non-zero complex point $\mathbf{x} \in \mathbb C^{n+1} \setminus \{\mathbf{0}\}$ 
where $\mathbf{f}^\natural(\mathbf{x}) = \mathbf{0}$.
Then the system $\{f_1,\ldots, f_R\}$ satisfies the assumptions of Definition~\ref{stronghardylittlew}.
In particular, Theorem~\ref{thm:vachms} holds 
for $\{f_1,\ldots,f_R\}$. \end{corollary}
This is demonstrated
 using results from 
Rydin Myerson 
\cite{simon1,simon2,simon3,simon4}
who proved  the Hasse principle for the variety 
$f_1=\cdots=f_R=0$ via the circle method.
We recall his work here. 
Let $f_1, \ldots, f_R \in \mathbb Z[x_0, \ldots, x_n]$ be as in Corollary \ref{cor:birch}. 
The following result is implicitly proven 
in \cite[Theorem 1.1.1]{simon3} 
and the lecture notes \cite{simon4}, modulo a few standard adjustments. 
Specifically, the original theorem in \cite{simon3} only considers representations of $\mathbf{0}$,
deals with homogeneous rather than inhomogeneous polynomials, and excludes a Zariski-closed set of 
polynomials instead of non-singular systems. The first two differences are readily addressed using
the approach from \cite[Propositions 2.1 \& 3.1]{simon1}
which handles non-homogeneous polynomials.
The final difference regarding non-singularity was resolved 
by Rydin
Myerson in his  lecture notes \cite{simon4}. 
As a result one obtains the 
asymptotic \eqref{eq_simon_asymptotic},
which    improves     significantly on  the      
number of variables compared to the classical work of Birch~\cite{MR0150129}.  
Furthermore,~\eqref{eq:singseries} and \eqref{eq:singint}
are verified respectively 
in \cite[Lemmas 2.5 \& 2.6]{simon1}.
\subsection{Diagonal polynomials}
Let $A$ be an integer $R\times (n+1)$ matrix 
and asume that $n+1\geq R$. In \cite[Definition 1.2]
{MR4988591}
Rome--Yamagishi defined 
the constant 
$\Psi(A)$ as the maximum number of column-disjoint 
  $R \times R$ invertible submatrices
  formed by columns of $A$.
Note that if 
 the matrix is `generic' then 
 $\Psi(A)= [(n+1)/R]$.
When the polynomials are diagonal, 
Theorem \ref{thm:vachms} applies to systems in a small number of variables:
\begin{corollary}\label{cor:rome}
Assume that $d\geq 2 ,n,R$ are strictly positive 
integers satisfying $n+1\geq R$. 
Let $f_1,\ldots, f_R$ 
be diagonal integer forms of degree $d$
such that the matrix $A$
of their coefficients 
satisfies 
$$ \Psi(A) \geq \min\{2^d, d(d + 1)\} .$$
Then the assumptions of Definition~\ref{stronghardylittlew} hold for 
the system $\{f_1,\ldots, f_R\}$.
In particular, Theorem~\ref{thm:vachms} holds for 
$\{f_1,\ldots,f_R\}$. \end{corollary}
This is demonstrated
by recent work of 
Rome--Yamagishi \cite[Theorem 1.4]{MR4988591}.
As a special case,
for generic integer forms  of degree $d\geq 5$ in 
 $n+1$ variables,
Theorem~\ref{thm:vachms} holds for diagonal forms
when 
$$ n+1 \geq d(d + 1)R+R .$$ 
\begin{remark}This range offers a significant improvement over that in Corollary \ref{cor:birch}. Note that Theorems \ref{thm:mainthrm}, \ref{ref:chowla}, \ref{ref:batman}, \ref{ref:ritchmarsh}, and \ref{singular_hasse_principle} are proved for general polynomials as in Corollary \ref{cor:birch}. However, because their proofs rely solely on Theorem \ref{thm:vachms}, these results also hold in fewer variables for diagonal polynomials of the form given in Corollary \ref{cor:rome}.\end{remark}
 \subsection{A uniform upper-bound} 
\label{s:mates 1000pounds}
We    start with establishing a    result 
giving precise upper bounds  
without knowledge of $k$ on   progressions.

\begin{lemma}\label{corl:uperbnd}  
For $f_i$ and $s_i, b, \c B $   as in Theorem~\ref{thm:vachms}
and  any   $k:\N^R \to \mathbb C$, 
$P\geq 1 $
we have  
\[ 
\sum_{\substack{ \b t \in \Z^{n+1}\cap P\c B \\ 
\min_j s_j f_j(\b t ) > 0 }} 
 k(s_1 f_1(\b t), \ldots, s_R f_R(\b t) )  
 \ll 
P^{n+1-Rd}
\sum_{\boldsymbol \nu \in \N^R \cap [1 , b P^d]^R} 
|k(\boldsymbol \nu) | ,\] where the implied constant 
 is independent of $P$ and $k$.
\end{lemma}  
\begin{proof} 
Note that for   $\b t \in P\c B$ we have  
$ |f_i(\b t)| \leq b P^d$, hence,
\beq{eq:broschi}{
\sum_{\substack{ \b t \in \Z^{n+1}
\cap
P\c B \\ \min_j s_j f_j(\b t ) > 0 }}  k(s_1 f_1(\b t), \ldots, s_R f_R(\b t) ) 
=  
\sum_{\boldsymbol \nu \in \N^R \cap [1 , b P^d]^R} k(\boldsymbol \nu) \#\{\b t \in \Z^{n+1} \cap P\c B  :  f_j(\b t)= s_j  \nu_j \, \forall j \}.}The cardinality
in the right hand-side is 
$\ll P^{n+1-Rd} \mathfrak J
\mathfrak S +P^{n+1-Rd-\delta }$
by \eqref{eq_simon_asymptotic}.
By \eqref{eq:singseries}-\eqref{eq:singint}
both $\mathfrak S, \Psi$ converge absolutely and are 
bounded independently of $\boldsymbol \nu$. 
Hence, ~\eqref{eq:broschi} is 
\[ \ll P^{n+1-Rd}\sum_{\boldsymbol \nu \in \N^R \cap [1 , b P^d]^R} 
|k(\boldsymbol \nu) |.\qedhere\]
\end{proof}

\subsection{Proof of Theorem~\ref{thm:vachms}}
\label{s:pprrooff}
\begin{proof}
We adapt the   strategy as in the proof of the 
case $R=1$ from~\cite{destagnolR=1};
details are given only  when necessary. For clarity we denote 
 $\mathfrak S$ by 
$$\mathfrak S(\boldsymbol  \nu):=
\prod_{\substack{p \textrm{ prime}}} 
\lim_{m\to \infty} 
\frac{\#\{\b t \in (\Z/p^{m}\Z )^{n+1}:  \b f (\b t )
= \boldsymbol \nu \}}{p^{m(n+1-R) }}.$$ 
Recall $W_z$ from Definition~\ref{def:wzez}.
We first approximate the product by a truncated version 
$$\mathfrak S^\flat(\boldsymbol \nu):= 
\frac{\#\{\b t \in (\Z/W_z\Z )^{n+1}:  \b f (\b t )
= \boldsymbol \nu \}}{W_z^{n+1-R }} .$$ 
The succeeding result is proved similarly 
as~\cite[Lemma 2.1]{destagnolR=1} by using the estimate \eqref{eq:singseries}.
\begin{lemma}\label{lem:cramer}There exists $c'=c'(\b f )>0$ such that 
for all  $\boldsymbol \nu \in \Z^R$ we have 
$$\mathfrak S(\boldsymbol\nu)-\mathfrak S^\flat(\boldsymbol\nu)\ll  \widetilde{\epsilon}(z) +z^{-c'},$$ where the implied constant is independent of $\boldsymbol\nu$.
\end{lemma} 
We rewrite the  singular integral $\mathfrak J$ as   
$$\mathfrak J(\boldsymbol\mu):=\int_{\R^R} I(\c B;\boldsymbol\gamma)\mathrm e(-\langle  \boldsymbol\gamma , \boldsymbol\mu \rangle   ) \mathrm d \boldsymbol\gamma,
\hspace{0,5cm} I(\c B;\boldsymbol\gamma):=\int_{\c B} \mathrm e \l(\sum_{i=1}^R \gamma_i f_i^\natural(\b t ) \r )\mathrm d \b t, $$
 where $\langle \cdot , \cdot \rangle$ denotes the inner product in $\R^R$.

We are now in position to start averaging the arithmetic function over the polynomial values.
\begin{lemma}
\label{lem:letsstart}
There exists $c''=c''(\b f )>0$ such that
for all large $P$ one has 
\begin{align*} P^{-n-1+Rd}
&\sum_{\substack{ \b t \in \Z^{n+1}
\cap P\c B \\ \min_j s_j f_j(\b t ) > 0 }}  k(s_1 f_1(\b t), \ldots, s_R f_R(\b t) ) = O( 
( P^{-\delta} +  \widetilde{\epsilon}(z)  + z^{-c''}  
)  \|k\|_1)
\\ & 
+ W_z^{-n-1+R} \sum_{\b t \in (\Z/W_z \Z)^{n+1} }  \int_{\R^R  } I(\c B;\boldsymbol{\gamma})
 \Bigg\{ \sum_{ \substack{ \boldsymbol{\nu}\in \N\cap[1,b P^d]^R \\ s_j\nu_j \equiv f_j(\b t ) \md{W_z} \, \forall j }} k(\boldsymbol{\nu})  
  \mathrm e(- \langle \boldsymbol{\gamma},\mathbf{s}\boldsymbol{\nu}\rangle P^{-d}) \Bigg\} \mathrm d \boldsymbol{\gamma},
\end{align*} where the implied constant depends only on $\b f $ and $\mathscr B$.
\end{lemma}
\begin{proof} Starting from~\eqref{eq:broschi}, the reasoning in~\cite[Lemma 2.2]{destagnolR=1} can be   adapted in a natural way. \end{proof}
The following multivariable version of partial summation   can   be proved by induction on $R$.
\begin{lemma}
\label{partsummation}
    Assume that we have a function $f: \N^R\to \mathbb C$
and $g_1,\dots,g_r: [1,\infty) \to \mathbb C$ be $\c C^1$. 
For  $x_1,\dots,x_R \in \R^R$ set 
$ 
F(x_1,\dots,x_R):=\sum_{n_i \leqslant x_i}f(n_1,\dots,n_R).
$
Then we have
  $$\sum_{\substack{ \b n \in \N^R \\ \forall j: n_j\leq x_j    } } f(\b n ) \prod_{j=1}^R g_j(n_j) 
 = \sum_{\substack{
 n_j \leqslant x_j, \, \forall j \in S 
 \\
S \subset\{1,\dots,R\}}}
(-1)^{\#S}
\l\{\prod_{ j \not\in S} g_j(x_j)\r\} T_S , $$ where $T_S $ is an integral over $\R^{\#S}$ that is given by 
$$ \int_{\substack{ \b t = (t_j)_{j\in S} \\ \forall j \in S: t_j \in [1, x_j] }}
F(x_j\mathds 1 (j \not\in S ) +t_j \mathds 1 (j \in S )) 
\prod_{ j \in S} g'_j(t_j) \mathrm d \b t .$$ 
\end{lemma}

The next step is   to perform a partial summation to deal with the   sum over $\boldsymbol \nu$ in Lemma~\ref{lem:letsstart}.

\begin{lemma}\label{lem:partsum}For   $x\geq 1$, $a_i,q\in \mathbb Z$ with $q\neq 0$ and any $\boldsymbol{\mu}\in \R^R$ and $x_1,\dots,x_R \in [1,x]$ we have 
$$
\sum_{ \substack{ \boldsymbol{\nu}\in \N^R, \, \nu_j \leqslant x_j \\ \nu_j \equiv a_j \md{q} \, \forall j }} k(\boldsymbol{\nu})  
  \mathrm e( \langle \boldsymbol{\mu},\boldsymbol{\nu}\rangle    )=
\rho( \textbf{a},q) \int_{ \prod_{j=1}^R[1,x_j]^R} \omega(\mathbf{t})\mathrm e (\langle\boldsymbol{\mu},\mathbf{t}\rangle) \mathrm d\mathbf{t}
+O\left( E(x;q)\sum_{ S \subset   [1,R]} x^{\#S}\prod_{ j \in S} |\mu_j  |\r)
,$$   where the implied constant is absolute.
\end{lemma}\begin{proof}
Take $f=k\cdot \mathds{1}_{\nu_i \equiv a_i \md q}$ and $g_i(y)=\mathrm e(-\mu_i y)$ in 
Lemma~\ref{partsummation}. Then 
$$
F(x_1,\dots,x_R)=\rho(\mathbf{a},q)G(x_1,\dots,x_R)+O(E(x;q))
,$$
where $G(x_1,\dots,x_R)=\int_{\mathbf{t} \in\prod_{i=1}^R [1,x_i]} \omega(\mathbf{t})\rm{d} \mathbf{t}.$ 
We have 
\begin{equation}
\frac{\partial^R G}{\partial x_1 \cdots \partial x_R}(x_1,\dots,x_R)=\omega(x_1,\dots,x_R).
\label{ddd}
\end{equation}
due to the fact that $\omega$ is $\mathcal{C}^1$ in Theorem~\ref{thm:vachms}.
We therefore obtain
$$ \sum_{\substack{ n_j \leqslant x, \, \forall j \in S \\
S \subset \{1,\dots,R\}}}
 (-1)^{\#S} \mathrm e\left( \sum_{j \not\in S}\mu_j x_j\right) T_S
$$
with
$$
T_S=\int_{\substack{ \b t = (t_j)_{j\in S} \\ \forall j \in S: t_j \in [1  , x] }}
 F(x\mathds 1 (j \notin S ) +t_j \mathds 1 (j \in S ))  
\mathrm e \l(\sum_{j\in S} \mu_j t_j \r) 
\prod_{ j \in S} (2 \pi i \mu_j  ) 
\mathrm d \b t.
$$
Substituting, we get
$$
\rho(\mathbf{a},q)\sum_{\substack{
n_j \leqslant x, \, \forall j \in S 
\\ S \subset \{1,\dots,R\}}} (-1)^{\#S} 
\mathrm e\left( -\sum_{j \not\in S}\mu_j\nu_j\right) A_S+O\left(E(x;q)
\sum_{\substack{ n_j \leqslant x, \, \forall j \in S \\
S \subset \{1,\dots,R\}}} B_S \right),
$$
where
$$
A_S=\int_{\substack{ \b t = (t_j)_{j\in S} \\ \forall j \in S: t_j \in [1  , x] }}
 G(x\mathds 1 (j \notin S ) +t_j \mathds 1 (j \in S ))  
\mathrm e \l(\sum_{j\in S} \mu_j t_j \r) 
\prod_{ j \in S} (2 \pi i \mu_j  ) 
\mathrm d \b t
$$
and $$
B_S=\int_{\substack{ \b t = (t_j)_{j\in S} \\ \forall j \in S: t_j \in [1 , x] }}
 \prod_{ j \in S}  (  |\mu_j  |)
\mathrm d \b t=x^{\#S}\prod_{ j \in S} |\mu_j  |.
$$ Using repeated partial summations together with \eqref{ddd}
concludes the proof.
\end{proof}
Define$$H= \int_{\R^R  } I(\c B;\boldsymbol{\gamma})
\left( \int_{[P^{-d},b]^R} \omega(\mathbf{t})
\mathrm e\left( -\langle 
\boldsymbol{\gamma},\mathbf{s}\mathbf{t}\rangle\right)\rm{d}
\mathbf{t}\right)\rm{d}\boldsymbol{\gamma}.$$

\begin{lemma}\label{lem:guitar}
There exists $c'''=c'''(\b f )>0$ such that
for all large $P$ one has 
\begin{align*} P^{-n-1}
\sum_{\substack{ \b t \in \Z^{n+1}\cap P\c B \\  s_j f_j(\b t ) > 0, \, \forall j }}  
k(s_1f_1(\b t), \ldots,s_R f_R(\mathbf{t)}) 
- &\frac{H }{W_z^{n+1-R}}  \sum_{\b t \in (\Z/W_z )^{n+1} } \rho((s_1f_1(\b t), \ldots,s_R f_R(\mathbf{t)})),W_z) 
\\  \ll &   (W_z/P^d)^R
E(bP^d;W_z)+ \frac{\|k\|_1 }{P^{Rd} }
\left( P^{-\delta} +  \widetilde{\epsilon}(z) 
+ z^{-c'''}  \right) .\end{align*}\end{lemma} 
\begin{proof}
Applying Lemmas~\ref{lem:letsstart}-\ref{lem:partsum}
together  with \eqref{eq:singint}
we infer that 
$$
\int_{\R^R  } |I(\c B;\boldsymbol{\gamma})|\left(\sum_{ S \subset \{1,\dots,R\}}\prod_{ j \in S} |\gamma_j  |\right)\rm{d}\boldsymbol{\gamma}<+\infty.
$$ The remaining of the proof 
can be conducted as in~\cite[Lemma 2.4]{destagnolR=1}.
\end{proof}
The proof of Theorem~\ref{thm:vachms}
can be concluded by using 
\textit{mutatis mutandis} the Fourier analysis approach used in the $R=1$ case 
in~\cite[Lemmas 2.5-2.8]{destagnolR=1}.\end{proof}

\subsection{Proof of Theorem~\ref{ref:ritchmarsh}}
\label{prooftitchm}
In this section we prove 
Theorem~\ref{ref:ritchmarsh}. For this we apply 
Theorem~\ref{thm:vachms} with $k(m_1,m_2)=
\mathds{1}_{\mathcal{P}}(m_1)\tau(m_2)$, where 
$\mathds{1}_{\mathcal{P}}$ is the indicator function of the primes. 
With this choice of $k$ the counting function 
in \eqref{eq:basicproperty} splits as  
$$
\sum_{\substack{ 1\leq m_i \leq x_i \\ \b m \equiv \b a \md q }} k(\b m )=
\sum_{\substack{  m_1 \leq x_1 \\ m_1 \equiv 
a_1 \md q }} \mathds{1}_{\c P}( m_1 )
\sum_{\substack{  m_2 \leq x_2
\\ m_2 \equiv a_2 \md q }} \tau( m_2 ).
$$ By the 
  Siegel--Walfisz theorem we may  write the sum over $m_1$ in the right hand-side as $$\frac{x_1}{\log x_1}
\left(\frac{\mathds 1(\gcd(a_1,q)=1)}{\phi(q)}+O\left(\frac{1}{\log x_1}\right)\right)$$ with  an absolute implied constant. As shown  in 
~\cite[\S 3.3]{destagnolR=1}, 
we   have 
$$ \sum_{\substack{  m_2 \leq x_2
\\ m_2 \equiv a_2 \md q }} \tau( m_2 )
= 
\frac{x_2}{q} 
\sum_{r\mid q }\frac{c_r(a_2)}{r}
\left(\log \left(\frac{x_2}{r^2}\right)+O(1) \right)
+O(x_2^{8/15} + q^{1/2} x_2^{1/5}),$$ 
where $c_r(a_2)$ is 
Ramanujan's sum and the implied constants are absolute.  
Using the trivial bound $|c_r(a_2)|   \leq r $ and 
assuming that $q\leq x_2 $ 
we obtain the estimate $$
\frac{x_2 (\log x_2)}{q} 
\sum_{r\mid q }\frac{c_r(a_2)}{r}
+O\left(
\frac{x_2}{q}\tau(q) (\log q) 
+
x_2^{8/15} + q^{1/2} x_2^{1/5}\right )
=\frac{x_2 (\log x_2)}{q} 
\sum_{r\mid q }\frac{c_r(a_2)}{r}+O(x_2),$$
with an absolute implied constant.

Hence, 
\begin{align*}
\sum_{\substack{  m_i \leq x_i \\
\b m \equiv \b a \md q }}
\mathds{1}_{\c P}(m_1)\tau(m_2)
 &=
\frac{x_1 x_2 (\log x_2)}{\log x_1}
\frac{\mathds 1(\gcd(a_1,q)=1)}{\phi(q) q}
\sum_{r\mid q }\frac{c_r(a_2)}{r}
+O\left( \frac{x_1x_2}{\log x_1} 
+\frac{x_1 x_2 (\log x_2) }{(\log x_1)^2}\right)
\\ &=
\rho(a_1,a_2,q) \int_{1}^{x_1}
\int_{1}^{x_2}
\omega(\b t )\mathrm d \b t 
+O\left( \frac{x_1x_2}{\log x_1} 
+\frac{x_1 x_2 (\log x_2) }{(\log x_1)^2}\right)
 \end{align*} holds   with 
 an  absolute implied constant    and   where $$\rho(a_1,a_2,q)= \frac{\mathds 1(\gcd(a_1,q)=1)}{\phi(q)q}
\sum_{r\mid q }\frac{c_r(a_2)}{r}, 
\ \ \ 
\omega(t_1,t_2)= \frac{\log t_2}{\log t_1}.$$   
We note that

$$
 \lim_{P\to \infty} 
 \int\limits_{\substack{  \c B  \\  
 \min_j s_j f_j(\b t ) > P^{-d} } }  
\frac{\log (P^d  f_2(\b t  ))}
{\log (P^d  f_1(\b t  ))}
\mathrm d \b t 
=
\lim_{P\to \infty} 
 \int\limits_{\substack{  \c B  \\  
 \min_j s_j f_j(\b t ) > P^{-d} } }   1
\mathrm d \b t 
= \mathrm{vol}(\c B)
$$ by~\cite[Lemma 1.19]{circle}.
We may  use this together with    Theorem~\ref{thm:vachms} 
in the same way as in ~\cite[\S 3.3]{destagnolR=1}.
The end result   is

 \[ \lim_{P\to \infty} P^{-n-1}
\sum_{\substack{ \b t \in \Z^{n+1} \cap P\mathcal B\\
f_1(\b t)   \textrm{ prime}   } } 
\tau(f_2(\b t  ) ) 
=\mathrm{vol}(\c B)
\lim_{z\to\infty}
 \sum_{\substack{ \b t \in (\Z/W_z\Z )^{n+1} \\ 
 \gcd(f_1(\b t ) ,W_z)=1}} 
 \frac{ 1}{ \phi(W_z) W_z^{n}  }\sum_{r\mid W_z }\frac{c_r(f_2(\b t))}{r} 
\] where $z$ is a function of $P$ that tends to infinity slowly with $P$ to ensure that $W_z\leq P$. 
By the Chinese remainder theorem one can see that the limit in the right-hand side  equals the quantity 
$\prod_{p}(1-1/p) ^{-1} \lim_{m\to \infty} \sigma_p(m)$, where  $$\sigma_p(m) = 
 \frac{ 1 }{   p^{m(n+1) }  }
\sum_{\substack{ \b t \in (\Z/p^m\Z )^{n+1} \\ 
p\nmid f_1(\b t )}} 
 \sum_{\ell=0}^m p^{-\ell} c_{p^{\ell}}(f_2(\b t) ). $$ 
 Since $c_1(f_2(\b t)) =1 $ we infer that 
 the contribution of the term $\ell=0$ is 
$$
 \frac{ 1 }{   p^{m(n+1) }  }
\sum_{\substack{ \b t \in (\Z/p^m\Z )^{n+1} \\ 
p\nmid f_1(\b t )}} 
1=  \frac{ 1 }{   p^{m(n+1) }  } p^{(m-1)(n+1)}
\sum_{\substack{ \b t \in (\Z/p\Z )^{n+1} \\ 
p\nmid f_1(\b t )}} 
1=  \frac{ 1 }{   p^{n+1 }  }  
\sum_{\substack{ \b t \in (\Z/p\Z )^{n+1} \\ 
p\nmid f_1(\b t )}} 
1 . $$  Recalling the definition of the Ramanujan sum
allows us to write the contribution of $\ell\geq 1 $ as 
\begin{align*}
& \frac{ 1 }{   p^{m(n+1) }  }
\sum_{\substack{ \b t \in (\Z/p^m\Z )^{n+1} \\ 
p\nmid f_1(\b t )}} 
 \sum_{\ell=1}^m p^{-\ell} \sum_{b\in (\Z/p^{\ell} \Z)^{\ast}}
 \mathrm e (p^{-\ell} bf_2(\b t))
  \\ =
 &\frac{ 1 }{   p^{m(n+1) }  }
 \sum_{\substack{  1\leq \ell \leq m  \\ b\in (\Z/p^{\ell} \Z)^* }} 
 p^{-\ell} 
 \sum_{\substack{ \b t \in (\Z/p^m\Z )^{n+1} \\ 
p\nmid f_1(\b t )}}  \mathrm e (p^{-\ell} bf_2(\b t))
=
 \sum_{\substack{  1\leq \ell \leq m  \\ b\in (\Z/p^{\ell} \Z)^* }} 
\frac{ 1 }{   p^{\ell(n+2) }  }
 \sum_{\substack{ \b t \in (\Z/p^{\ell}\Z )^{n+1} \\ 
p\nmid f_1(\b t )}}  \mathrm e (p^{-\ell} bf_2(\b t))
.\end{align*} We thus obtain  $$ \sigma_p(m)=   p^{-(n+1) }
\sum_{\substack{ \b t \in (\Z/p\Z )^{n+1} \\ 
p\nmid f_1(\b t )}}  1
+ \sum_{\ell=1}^m   p^{-(n+2) \ell}
\sum_{b\in (\Z/p^{\ell} \Z)^*}
\sum_{\substack{ \b t \in (\Z/p^{\ell}\Z )^{n+1} \\ 
p\nmid f_1(\b t )}}   \mathrm e (bf_2(\b t )p^{-\ell}).$$

This resembles the 
  standard exponential sum for the singular series of 
$f_2=0$ except that the factor $p^{-(n+2) \ell}$ is off. 
To work around this we will write  $\sigma_p(m)$ as 
a telescopic sum. Firstly, we note that 
one can easily 
express   
$\varpi_m:=
\#\{\b x \in (\Z/p^m\Z)^{n+1}:p\nmid f_1(\b x ), p^m \mid f_2(\b x) \}  $ as   $$
\frac{\varpi_m}{p^{nm}}
=\frac{\#\{ \b x \in (\Z/p\Z)^{n+1}:p\nmid f_1(\b x )\}}{p^{n+1}} 
+  
\sum_{k=1}^m
p^{-(n+1)k}
\sum_{a \in (\Z/p^k\Z)^* }
\sum_{\substack{ \b x \in (\Z/p^k\Z)^{n+1}
\\p\nmid f_1(\b x )}}
\mathrm e(-ap^{-k} f_2(\b x) )
,$$ for example, as in the last equation of 
\cite[page 259]{MR0150129}.
Hence, 
 for $m\geq 1$ we get 
$$
\frac{\varpi_{m+1}}{p^{n(m+1)}}
-\frac{\varpi_m}{p^{nm}}=
p^{-(n+1)(m+1)}
\sum_{a \in (\Z/p^{m+1}\Z)^* }
\sum_{\substack{ \b x \in (\Z/p^{m+1}\Z)^{n+1}
\\p\nmid f_1(\b x )}}
\mathrm e(-ap^{-m-1} f_2(\b x) ).
$$ Using this we can
write  $\sigma_p(m)$ as 
\begin{align*}
p^{-(n+1)  } 
&\sum_{\substack{ \b t \in (\Z/p\Z )^{n+1} \\ 
p\nmid f_1(\b t )}}  1
+
p^{-(n+2) }
\sum_{b\in (\Z/p \Z)^*}
\sum_{\substack{ \b t \in (\Z/p\Z )^{n+1} \\ 
p\nmid f_1(\b t )}}  
\mathrm e (bf_2(\b t )p^{-1})
+ \sum_{r=2}^m   p^{- r}
\left(\frac{\varpi_{r}}{p^{nr}}-\frac{\varpi_{r-1}}{p^{n(r-1)}}\right)
\\ &= 
\left(1-\frac{1}{p}\right)
\frac{\#\{\b x \in \F_p^{n+1}: f_1(\b x )\neq 0 \}}{p^{n+1}}
+ \frac{\varpi_{m+1}}{p^{1+(m+1)(n+1)}} 
+
\left( 1 -\frac{1}{p} \right)
\sum_{j=1}^m
\frac{\varpi_j}{p^{j(n+1)}} 
.\end{align*} Recalling the definition of 
$\tau_{p^m}$ given just before 
Theorem~\ref{ref:ritchmarsh}
we can  write this as  
 \begin{align*}
 &\frac{\varpi_{m+1}}{p^{1+(m+1)(n+1)}} +p^{-m (n+1)}
\sum_{\substack{ \b x \in (\Z/p^m\Z)^{n+1} \\ 
p\nmid f_1(\b x ) } }
\sum_{j=0}^m \mathds 1 (p^j \mid  f_2(\b x) )
\\=
&\frac{1}{p^{1+(m+1)(n+1)}}
\frac{\varpi_{m+1}}{p^{1+(m+1)(n+1)}} 
+p^{-m (n+1)}
\sum_{\substack{ \b x \in (\Z/p^m\Z)^{n+1} \\ 
p\nmid f_1(\b x ) } } \tau_{p^m}(f_2(\b x )) 
.\end{align*} The bound 
$\varpi_m \ll p^{mn+\epsilon}$   
from \cite[Equation (2.5)]{destagnolR=1}
shows that the fraction 
goes to $0$, thus, concluding the proof of the required estimate for \eqref{eq:basicproperty}.
One can feed this into 
Theorem~\ref{thm:vachms} to conclude the proof of 
Theorem~\ref{ref:ritchmarsh} by using identical steps as in \cite[\S 3.3]{destagnolR=1}. 

\subsection{Proof of Theorem \ref{singular_hasse_principle}}
\label{No, non vedete mai} 
Before stating the main auxiliary 
lemma  we  
 give an intermediate result on $p$-adic densities. 
For a homogeneous polynomial
$F\in \Z[x_1,\ldots,x_R]$,
a prime $p$, an integer  $m\geq 1 $ and 
$\mathbf{t} \in (\mathbb{Z}/p^m\mathbb{Z})^R$ 
with $p^{m } \mid F(\b t )$ 
we  define 
\[
\sigma_p(\mathbf{t}, p^m) := \lim_{M \to +\infty} 
\frac{\#\{\mathbf{x} \in (\mathbb{Z}/p^M\mathbb{Z})^R : F(\mathbf{x}) \equiv 0 \bmod {p^M}, 
\mathbf{x} \equiv \mathbf{t} \bmod {p^m}\}}
{p^{M(R-1)}}\]  and let 
$\sigma_p=\sigma_p(\b 0, 1)$. 
\begin{lemma}\label{exei_to_aeraki}Assume that 
$F\in \Z[x_1,\ldots,x_R]$ is a 
of smooth form of degree $d_2$ over $\Q$.
\begin{enumerate}
\item Assume  that  $\ell$ and $m$ are integers with 
$\ell \geq 1$ and $ m \geq 2\ell -1$. Assume that 
for a prime $p$  there exists
$\b t \in (\Z/p^m \Z)^R$ that satisfies 
$p^m \mid F(\b t )$ and  
$ \min_{j } v_p\Big( \frac{\partial F}{\partial x_j}(\b t)\Big)=\ell-1 $.
Then  for all $M\geq m+1$ we have  $$ 
\#\l\{ \b y \in (\Z /p^{M }\Z)^R :
F(\b y ) \equiv 0 \md{p^M}, 
\b y \equiv \b t \md {p^ m}\r \}  \geq p^{(R-1) (M-m) }
,$$ and hence, 
$ \sigma_p(\b t,p^m) \geq p^{-m (R-1)   }$.
\item 
There exists $C=C(F)>0$ such that if $p>C$ then for all $m \geq 1$ and $\b t \in (\Z/p^m \Z)^R$ with $p^m \mid F(\b t )$ we have
$$ \sigma_p(\b t , p^m ) \geq  \frac{1}{p^{m(R-1) }} .$$
\end{enumerate} 
\end{lemma}

\begin{proof} The proof of (1) is a simple modification of the induction
argument in \cite[Lemma 17.1]{davenport},
where the only difference lies in
starting induction at a higher step.
Before giving the details we note that all arguments in the proof 
of \cite[Lemma 17.1]{davenport}
use Taylor expansion and hence work for forms of abitrary degree.
Furthermore, assumptions \cite[(17.3) \& (17.4)]{davenport}
are met for $\b x =\b t $ due to  
$ \min_{j } v_p\Big( \frac{\partial F}{\partial x_j}(\b t)\Big)=\ell-1 $. Instead of   
assumption \cite[(17.2)]{davenport}
we now 
have 
$p^{2\ell-1+\nu_0} 
\mid F(\b t )$  
where $\nu_0$ is defined by $m= 2\ell-1+\nu_0$. The integer $\nu_0$ is the basis of our 
induction. 
The proof of \cite[Lemma 17.1]{davenport} proves with induction  
that for each $\nu\geq 0 $ every solution $\b x
\md {p^{2\ell-1+\nu}}$ lifts to $p^{R-1}$ different  
$\b y \md {p^{2\ell+\nu}}$ satisfying 
$$ F(\b y ) \equiv 0  \md {p^{2\ell+\nu}}, 
\b y \equiv \b t \md {p^{\ell+\nu}}.$$
Here we can start at $\nu=\nu_0$ and then for each $k\geq 0 $ proceed to find that there are   
at least $p^{(R-1) (k+1)} $ different   
$\b y \md  {p^{2\ell+k +\nu_0} }$
satisfying   $  F(\b y ) \equiv 0  \md 
{p^{2\ell+k +\nu_0} }$ and 
$ \b y \equiv \b t \md {p^ { 2\ell-1+\nu_0}}$.
In particular, letting $M= 2\ell+k+\nu_0$ we  
 obtain    $$ 
\#\l\{ \b y \in (\Z /p^{M }\Z)^R :
F(\b y ) \equiv 0 \md{p^M}, 
\b y \equiv \b t \md {p^ m}\r \}  \geq p^{(R-1) (M-m) }
.$$ 

Now, (2) follows directly from the first point 
with $\ell=1$
by noting that since $F$ is smooth over~$\mathbb  Q$, the variety $F=0$ is smooth over $\mathbb F_p$ for all $p\gg_F 1$. \qedhere
\end{proof}  
As is well-known, 
when counting integer solutions to $F(\b t )=0$ 
with the contraints $|t_j|\leq B_j $ for all $j$,
the       singular integral  is 
\begin{equation}\label{leon}
J(\b B) =\int_{-\infty}^\infty 
\left( \int_{\prod_{i=1}^R[-B_i,B_i]}  \mathrm e (\gamma F(\b t ))
\mathrm d \b t\right)\mathrm d \gamma= 
\lim_{\epsilon \to 0^+}
\frac{1}{2\epsilon} \int_{\prod_{i=1}^R[-B_i,B_i]} 
\mathds 1 ( |F(\b t )| \leq \epsilon) 
\mathrm d \b t \end{equation} for 
$ \b B \in [1,\infty)^R.$
\begin{lemma}\label{lem:baz} Let $F
\in \mathbb Z[x_1,\ldots, x_R]$ be 
 a non-singular homogeneous polynomial
 of degree $d_2$ with 
with $R>(d_2-1) 2^{d_2}$.
There exist $M, \eta > 0$ that only depend on $F$ such that for all $B_1,\ldots,B_R\geq 1$
with $\min_i B_i > \max_i B_i^{1-\eta}$
and all $q\in \mathbb N$,
$\mathbf{a} \in (\mathbb{Z}/q\mathbb{Z})^R$,
we have 
$$
\# \left\{ \mathbf{x} \in \mathbb{Z}^R : \ 
\begin{aligned}
&\lvert x_i \rvert \leq B_i  \forall i, \ F(\mathbf{x}) = 0, \\
&\mathbf{x} \equiv \mathbf{a} \bmod q
\end{aligned}
\right\}
-J(\b B) \prod_{p \mid q} \sigma_p (\mathbf{a}, p^{v_p(q)}) 
\prod_{p \nmid q} \sigma_p
\ll q^{M}  (\max_i B_i)^{R-{d_2}-\eta} 
$$
where the implied constant depends at most on $F$.  \end{lemma}
\begin{proof}The statement is  
the special case of 
\cite[Lemma 4.4]{elbaz}   corresponding to a hypersurface.  The only differenence is that 
the condition $|x_i|\leq B$ 
is replaced by $|x_i|\leq B_i$. To deal with this 
we draw upon the work of Schindler--Sofos
\cite[Theorem 2.1]{MR3867326} which deals with exactly such lopsided boxes. Note that \cite[Theorem 2.1]{MR3867326} uses smooth weights but these  are easy to remove, especially as we are not interested in optimising the error term.
\end{proof}

The error term in the present lemma has   non-explicit values for $\eta,M$ and   a   weaker error term
than the one  in   \cite[Theorem 2.1]{MR3867326}, however, 
it suffices for the applications.\\
\newline
\indent
We define the function  
$k:\Z^R \to \mathbb R$ given by 
$k(m_1,\ldots,m_R)=\mathds 1_{\{0\}}(F(\b m) )
$. We shall apply  
Theorem~\ref{thm:vachms}
to this function even though $k$ is not defined on $\mathbb N^R$. It is straightforward
to prove a version analogous to  Theorem~\ref{thm:vachms}
for functions defined on $\mathbb Z^R$
instead of  $\mathbb N^R$ simply by summing over all possible choices for the signs $s_1,\ldots, s_R$. This 
  would result in replacing  the sum in 
  \eqref{eq:basicproperty}  by 
$$
\sum_{\substack{ \b m \in \Z^R
, |m_i|\leq x_i \forall i 
\\\b m \equiv \b a \md q }}k(\b m )
.$$Let $b=\max\{|f_i(\b t)|: i\geq 1, |\b t |\leq 1\}$ 
so that if $\b m = \b f (\b t )$ and $\b t \in P[-1,1]$ then 
$|\b m| \leq B:= b P^{d_2}$. By Lemma \ref{lem:baz}  there are
$\eta,M>0$ depending only on $F$ such that for all  $B_1,\ldots, B_R$ in $[B^{1-\eta},B]$ we have 
\begin{equation}\label{Tommaso Traetta}
\sum_{\substack{ \b m \in \Z^R
, |m_i|\leq B_i \forall i 
\\\b m \equiv \b a \md q }}k(\b m )=
\rho(\b a , q)
\int_{\prod_{i=1}^R[-B_i,B_i]} \omega(\b t ) \mathrm d \b t
+O( q^{M}  B^{R-d_2-\eta} )
,\end{equation}
where the implied constant depends only on $F$
and where  
$$\rho(\b a , q):=\prod_{p\mid q } 
\frac{\sigma_p \left(\mathbf{a}, p^{v_p(q)}\right)}{\sigma_p}, \ \ \ 
\omega(\b t):= 
\Big( \prod_{p=2}^\infty\sigma_p \Big) 
\frac{\mathds 1 ( |F(\b t )| \leq B^{-1})  }{2 B^{-1}}
. $$ Note that we have taken $\epsilon=B^{-1}$ in \eqref{leon}
at the cost of introducing an error term of the form
$B^{-\eta'}$ where $\eta'=\eta'(F)>0$. This error term can be absorbed into the existing one by possibly taking a worse value for $\eta$. Furthemore, we   assumed
that $\sigma_p\neq 0$
because, since $F(\b f (\b t) )=0$ is everywhere locally soluble, 
the smooth sypersurface $F=0$ is also everywhere locally soluble, thus, $\sigma_p\neq 0$ by Hensel's lemma. Furthermore, the integral over $\gamma$ is truncated to the range $|\gamma|
 \leq B$ at the cost of a possibly smaller value for     $\eta$. 
 To see why we   need  this   estimate only when   
 $B_1,\ldots, B_R\in [B^{1-\eta},B]$
  we use Lemma  \ref{lem:baz}
  to get $$\#\{\b m \in \Z^{n+1}\cap [-B,B]^{n+1}:
F(\b m )=0,
|m_1| < B^{1-\eta} \} 
  \ll  B^{R-d_2-\eta'} $$ 
for some $\eta'(F)>0$. This can be absorbed into the error term of \eqref{Tommaso Traetta}.

Applying Theorem~\ref{thm:vachms} we thus infer that for any  box 
$\c B\subset \R^{n+1}$   as in Definition~\ref{def:box} one has
\begin{equation}\label{bus20}
\frac{\#\{\b  t \in \mathbb Z^{n+1}\cap P\mathcal B: 
F(\b f (\b t ) )=0\}}{P^{n+1}} \sim 
 \sum_{\b t \in (\Z/W_z\Z )^{n+1} } 
  \frac{  \rho(  \b f (\b t  ) ,W_z) }{ W_z^{{n+1}-R} }
  \! \! 
 \int\limits_{\substack{ \b t \in \c B \\ 
|f_j(\b t )| > P^{-d_2} \ \forall j} }  
\omega ( P^{d_1}  \b f(\b t  ) )  \mathrm d \b t 
\end{equation} as  $P\to \infty $. 
The error term in \eqref{Tommaso Traetta} is satisfactory as long as 
$q\leq B^\omega$ where    $\omega=\omega(F)>0$.
Hence, 
we  let $z= [\log \log B ] $ 
and for each prime $p$ we take 
$m_p(z)=  z $ so that  
by the prime number theorem we have 
$$W_z=\prod_{p\leq z } p^{m_p(z)}=
\prod_{p\leq z } p^{\sqrt{\log B} }
\leq 4^{z^2 } =B^{1/\sqrt{\log B} }
$$ as $z \to \infty$.
This ensures that $q=W_z\leq B^\omega$, thus, making sure that 
the 
error term in \eqref{Tommaso Traetta} with $q=W_z$ is negligible.
To complete the proof it suffices to show that if $F(\b f (\b t))=0$ is everywhere locally soluble then 
the right-hand side of \eqref{bus20} stays strictly 
positive as $P\to \infty$.

The integral equals \begin{align*}
\int\limits_{\substack{ \b t \in \c B \\ 
|f_j(\b t )| > P^{-d_2} \ \forall j} }  
\omega ( P^{d_1}  \b f(\b t  ) )  \mathrm d \b t 
&=\Big( \prod_{p=2}^\infty\sigma_p \Big) 
\int\limits_{\substack{ \b t \in \c B \\ 
|f_j(\b t )| > P^{-d_2} \ \forall j} }   
 \frac{\mathds 1 ( |F(  P^{d_1}  \b f(\b t  )     )| \leq B^{-1})  }{2 B^{-1}}
  \mathrm d \b t \\
&=P^{ - d_1d_2} \Big( \prod_{p=2}^\infty\sigma_p \Big) 
\int\limits_{\substack{ \b t \in \c B \\ 
|f_j(\b t )| > P^{-d_2} \ \forall j} }   
 \frac{\mathds 1 ( |F(    \b f(\b t  )     )| 
 \leq \epsilon') }{2 \epsilon'}
  \mathrm d \b t,
  \end{align*}  where $\epsilon'= (BP^{d_1d_2})^{-1}$.
Assuming that $F(\mathbf{f}(\mathbf{t})) = 0$ has point in $\mathbb R$, we can take  $\mathcal{B}$ such that this point is
in its interior. Owing to the smoothness of 
$F$ and $\mathbf{f}$, classical arguments 
(for example, \cite[Lemma 1.20]{circle}) ensure
that the real density (expressed as the integral over $\mathbf{t}$) is strictly positive, independently of $P$.
  Thus, putting everything together, we have shown that \begin{equation}\label{tassistaperparto}\frac{
\#\{\b  t \in \mathbb Z^{n+1}\cap P\mathcal B: 
F(\b f (\b t ) )=0\} }{ P^{n+1-d_1d_2}}\gg  
    \prod_{p\leq z  } 
\frac{   1 }{ p^{ m_p(z) (n+1-R)} } 
 \sum_{\b t \in (\Z/p^{m_p(z)}\Z )^{n+1} } 
\sigma_p (  \b f (\b t ), p^{m_p(z)} )
,\end{equation}
  with an implied constant that   depends only on 
  $\c B,F$ and $\b f $.
Now the goal is to show that if 
$F(\b f (\b t))=0$ is  locally soluble at $\Q_p$ for all  primes 
$p \leq z $ then 
the right-hand side of \eqref{tassistaperparto} stays strictly 
positive as $P\to \infty$.

Recalling the constant $C$ in 
the second part of Lemma \ref{exei_to_aeraki} we note that of a prime $p$ exceeds $C$ then  
 the $p$-adic term is  
$$ \geq 
\frac{   1 }{ p^{ m_p(z) (n+1-R)} } 
 \sum_{ \substack{\b t \in (\Z/p^{m_p(z)}\Z )^{n+1} 
 \\  F(\b f(\b t) )  \equiv 0 \md { p^{m_p(z) }}
 }} 
\frac{1}{p^{m_p(z)(R-1)}}
=
\frac{ \# \{\b t \in (\Z/p^{m_p(z)}\Z )^{n+1}: 
F(\b f(\b t ))   =0  \}  }{ p^{ m_p(z) n } }  
.$$ We can rewrite the numerator as 
$$\sum_{\substack{  \boldsymbol \nu \in (\Z/p^{m_p(z)} \Z)^R\\ F(\boldsymbol \nu) = 0  } } 
\#\{\b t \in (\Z/p^{m_p(z)}\Z )^{n+1}: \b f (\b t )=
\boldsymbol \nu\} \geq 
\sum_{\substack{  \boldsymbol \nu \in (\Z/p^{m_p(z)} \Z)^R\\ F(\boldsymbol \nu) = 0  } } p^{{m_p(z)}(n+1-R) } 
(1-p^{-1-\beta}) $$
by \eqref{eq:singseries}, where $\beta=\beta(\b f ) > 0 $.
Hensel's lemma  for $F=0$ shows that  for all $m\geq 1 $ we have 
$$ 
\frac{ \#\{\boldsymbol \nu  \in (\Z/p^{m}\Z)^R: F( \boldsymbol \nu )=0 \}}{p^{m(R-1)}} \geq 
\frac{ \#\{\boldsymbol \nu  \in (\Z/p\Z)^R: F( \boldsymbol \nu )=0 \}}{p^{R-1}}
 \geq  1-p^{-1-\beta'} 
$$ for some $\beta'=\beta'(F)>0$, hence, for $p\gg _F 1 $ we get 
$$ \frac{ 1 }{ p^{ m_p(z) (n+1-R)} } 
 \sum_{\b t \in (\Z/p^{m_p(z)}\Z )^{n+1} } 
\sigma_p (  \b f (\b t ), p^{m_p(z)} )
\geq (1-p^{-1-\gamma } )^2  
$$ where $\gamma = \max\{ \beta,\beta' \}$.
We have shown that there is $C=C(F)>0$ such that 
$$\frac{\#\{\b  t \in \mathbb Z^{n+1}\cap P\mathcal B: 
F(\b f (\b t ) )=0\} }{ P^{n+1-d_1d_2}}
\gg  
    \prod_{p\leq C  } 
\frac{   1 }{ p^{ m_p(z) (n+1-R)} } 
 \sum_{\b t \in (\Z/p^{m_p(z)}\Z )^{n+1} } 
\sigma_p (  \b f (\b t ), p^{m_p(z)} )
,$$
where the implied constant is independent of $P$ and $z$.
Recalling that for each prime $p\leq z$ we take $m_p(z)=z$ 
allows us to rewrite this as \begin{equation}\label{betercolsaul}\frac{
\#\{\b  t \in \mathbb Z^{n+1}\cap P\mathcal B: 
F(\b f (\b t ) )=0\} }{ P^{n+1-d_1d_2}}\gg  
    \prod_{p\leq C  } 
\frac{   1 }{ p^{ z(n+1-R)} } 
 \sum_{\b t \in (\Z/p^z\Z )^{n+1} } 
\sigma_p (  \b f (\b t ), p^z )
.\end{equation}

We write the $p$-adic term   as 
$$\frac{   1 }{ p^{ z (n+1-R)} } 
\sum_{\substack{ \b y \in (\Z/p^z\Z )^{R}  
\\ F(\b y ) \equiv 0 \md {p^z} }}
 \Bigg(
\sum_{\substack{ \b t \in (\Z/p^z\Z )^{n+1} \\
\b f (\b t ) 
\equiv \b y \md {p^z }  }} 
1\Bigg) \sigma_p( \b y , p^z ) .
$$ By restricting to those $\b y $ with 
$ \min_{j } v_p\Big( \frac{\partial F}{\partial x_j}(\b y)\Big)\leq (m_p(z)-1)/2 $ and using the first part of Lemma~\ref{exei_to_aeraki}
we obtain the lower bound 
\begin{equation}
\label{antigona:traetta}
\geq 
\frac{   1 }{ p^{ z n } } 
\Osum_{\substack{ \b y \in (\Z/p^z\Z )^{R}   
\\ F(\b y ) \equiv 0 \md {p^z } }}
 \Bigg(
\sum_{\substack{ \b t \in (\Z/p^z\Z )^{n+1} 
\\ \b f (\b t ) 
\equiv \b y \md {p^z }  }} 
1\Bigg)  ,
\end{equation}
where $\Osum$ is subject to the condition $ \min_{j } v_p\Big( \frac{\partial F}{\partial x_j}(\b y)\Big)
\leq (z-1)/2 $.

Now assume that for a prime $p\leq C$ 
the hypersurface $F(\b f (\b t))=0$ has a point in $\Q_p$.
By homogeneity there is $\b r \in \Z_p^{n+1}$ such that 
$ p \nmid  \b r $  and $F(\b f (\b r))=0$ in $\Z_p$. Among such $\b r $ we select   one which minimises the quantity $ \kappa:=
\min_j v_p(f_j(\b r ) )$, hence, 
$\kappa:= \kappa(p)\geq 0$. In other words, $p^\kappa \mid \b f (\b r )$ and $p^{\kappa+1} \nmid  \b f (\b r )$.
For all $i$ we have $$
\frac{\partial F}{\partial x_i} (\b f (\b r ) )
=\frac{\partial F}{\partial x_i} \l(p^\kappa 
\frac{\b f (\b r ) }{p^\kappa} \r) = p^{\kappa (d_2-1)}
\frac{\partial F}{\partial x_i} \l(  
\frac{\b f (\b r ) }{p^\kappa} \r).
$$
Since $F$ is smooth over $\mathbb Q$
we can find for every prime $p$ an integer $N_p$ such that 
for all $\b z \in \Z_p^{n+1}$ with $p \nmid \b z $ 
one has $ \min_i v_p\left(\frac{\partial F}{\partial x_i} 
(\b z )\right) \leq N_p$. Indeed, if the derivatives at these points were unbounded one can construct a converging sequence $\b z_k \in \mathbb Z_p^{n+1}$ with  $p \nmid \b z_k$ and $\min_i v_p\left(\frac{\partial F}{\partial x_i} 
(\b z_k )\right) \geq k$ and the limit $z =\lim z_k$ is a singular point on $F=0$, which does not exist by assumption. Furthermore, $N_p=0$ suffices for all but finitely many primes, which we have used implicitly at the beginning of the subsection, and will not require further.

Applying this with $\b z = \b f (\b r ) p^{-\kappa}$ we infer that 
$$\min_i v_p\l(\frac{\partial F}{\partial x_i} 
(\b f (\b r ) )\r) \leq \kappa (d_2-1)+N_p.$$
Note that $\kappa$ and $N_p$ are functions of $p$ and $F,\b f$, however, since  we are only considering $p\leq C=C(F)$, 
we infer that $\kappa$ is upper bounded by a function of $F$ and $\b f $ alone.
%
%
The conclusion is that 
there is $\ell=\ell(F,\b f)\geq 0$ such that for all primes $p\leq C$ there exists $\b r \in \Z_p^{n+1}$ with 
$\min_i v_p\left(\frac{\partial F}{\partial x_i} 
(\b f (\b r ) )\right) \leq \ell $ and $F(\b f (\b r))=0$. 
Since $z=[\log \log B]
\to \infty$, we certainly have $\ell \leq 
(z-1)/2$. Hence, restricting the sum in \eqref{antigona:traetta} to the term $\b t =\b r$ gives the lower bound $$
\frac{  \#\{\b y \in (\Z/p^z\Z)^R: 
F(\b y ) \equiv 0 \md{p^z},
\b y \equiv \b f (\b r )\md{p^z}\}}{ p^{ z n } } 
\geq p^{-zn}.$$
Hence, the overall bound becomes $$
\frac{
\#\{\b  t \in \mathbb Z^{n+1}\cap P\mathcal B: 
F(\b f (\b t ) )=0\} }{ P^{n+1-d_1d_2}}\gg  
    \prod_{p\leq C  } p^{-zn}={\gamma'}^{-z}
 $$ for some $\gamma'>1$ that depends only on $C$ and $n$. 
 Since $z=[\log \log B]$ and $B=b P^{d_2}$
 we have 
 ${\gamma'}^{z} \geq (\log P)^{\eta}$ for some $\eta>0$
 that depends only on $\gamma'$,
 $b$ and $ d_2$.
 We have therefore proved that if $F(\b f (\b x ))=0$ is everywhere locally soluble then  $$
\#\{\b  t \in \mathbb Z^{n+1}\cap P\mathcal B: 
F(\b f (\b t ) )=0\}  \gg  \frac{
 P^{n+1-d_1d_2} }{(\log P)^{\eta}}
  $$ 
 and in particular, $F(\b f (\b x ))=0$ has a non-trivial $\Q$-rational point. 

\section{Equidistribution       and  other 
analytic lemmas}
\label{s:sieglwalfsz}

In \S\ref{RV108} we prove that the arithmetic function 
that is   relevant to quadratic
norms is equidistributed in arithmetic progressions.
While this can be done for completely general arithmetic progressions, 
only special progressions are needed in the 
 application, namely,  we need only consider 
   progressions that are nearly coprime to
   the modulus, see~\eqref{eq:padrona}. This allows us to write the average of the multivariable
arithmetic function as a product of averages of arithmetic functions
  in a single variable and then apply   results from the literature regarding the latter sums.

At last, 
in \S\ref{s:arcturus is a most fucking awsome band}, 
we do establish certain bounds via the large sieve,
which, though requisite for the arguments of 
\S\ref{s:proofproofproof}, would cast an unwonted shadow 
upon the symmetry of exposition were they to intrude 
therein.
 
\subsection{Equidistribution}
\label{RV108}
To estimate $N(B)$ we shall   use Theorem~\ref{thm:vachms} in \S\ref{s:proofproofproof}.
This requires a result 
on equidistribution in arithmetic progressions
that we prove in this section. Recall that $D$ is a non-zero 
square-free integer other than $1$.  The next result is
from~\cite[Ch.III,Th.1]{serrecourse}.
\begin{lemma}\label{lem:hilbser}
For a prime $p$ and integers 
$\alpha,\beta,u,v$   with $p\nmid uv$ 
we have\[ 
(p^\alpha u, p^\beta v)_{\Q_p}=
\begin{cases}  
(-1)^{\frac{(u-1)(v-1)}{4}+\frac{\alpha (v^2-1) }{8}  +\frac{\beta (u^2-1)}{8}   }                        ,&\mbox{if } p=2, \\
(-1)^{\frac{ \alpha \beta (p-1)}{2}} \big(\frac{v}{p} \big)^\alpha   \big(\frac{u}{p} \big)^\beta    ,&\mbox{if } p\neq 2,  \end{cases}\]
where $\left(\frac{\cdot }{p}\right)$ is the Legendre quadratic symbol. \end{lemma}

As usual we let $v_p$ denote the standard $p$-adic valuation.
The next two lemmas are   applications of Lemma~\ref{lem:hilbser}.
\begin{lemma}
\label{lem:detectorslemma}
For $a\in \Z\setminus \{0\}$  
the conic $x_0^2-D x_1^2=a x_2^2$ has a $\Q$-point 
if and only if 
all of the following conditions hold,
\begin{enumerate}
\item $D<0\Rightarrow a >0$;
\item  $p\nmid 2D \ \mbox{and} \ \left(\frac{D}{p}\r)=-1 \Rightarrow 2\mid v_p(a)$;
\item  $p\mid D \ \mbox{and} \ p\neq 2  \ \mbox{and} \  2\mid v_p(a) 
\Rightarrow \left(\frac{ap^{-v_p(a) } }{p}\r)=1$;
\item  $p\mid D \ \mbox{and} \ p\neq 2  \ \mbox{and} \  2\nmid v_p(a)  
\Rightarrow \left(\frac{ap^{-v_p(a) } }{p}\r)=\left(\frac{-D/p}{p}\right)$;
\item  the conic $x_0^2-D x_1^2=a x_2^2$ has a point in $\Q_2$.
\end{enumerate}
\end{lemma}

\begin{lemma}
\label{lem:detectorslemma2}
Let $a\in \Z\setminus \{0\}$ be square-free.
Then the conic $x_0^2-D x_1^2=a x_2^2$ has a point in $\Q_2$
if and only if   
\begin{itemize} 
\item  $D\equiv 3 \md{8}    \Rightarrow a\equiv 1,5 \text{ or } 6 \md{8}$;
\item  $D\equiv 5 \md{8}    \Rightarrow a\equiv 1\md{2}$;
\item  $D\equiv 7 \md{8} \Rightarrow a\equiv 1 \text{ or } 2 \md{4}$;
\item  $D\equiv 2 \md{16}    \Rightarrow a\equiv   1,2,7,9,14   \text{ or } 15 \md{16}$;
\item  $D\equiv 6 \md{16}    \Rightarrow a\equiv   1,3,9,10,11 \text{ or } 14 \md{16}$;
\item  $D\equiv 10 \md{16}    \Rightarrow a\equiv 1,6,7,9,10   \text{ or } 15 \md{16}$;
\item  $D\equiv 14 \md{16}    \Rightarrow a\equiv 1,2,3,6,9     \text{ or } 11 \md{16}$.
\end{itemize}
\end{lemma}  

\begin{remark}\label{remIHP}
In particular, knowing 
$v_2(a)\md 2$ and $a2^{-v_2(a)} \md{16}$
fully determines $(D,a)_2$.
For odd primes $p$
knowing $v_p(a)\md 2$ and    $ap^{-v_p(a)} \md{p}$
fully determines $(D,a)_p$.
\end{remark}

Let $z>1$ be an integer and let  
$W_z\in \mathbb N$ be a square-full integer given by  
$$ W_z=\prod_{p\leq z}p^{z}.$$ 
  Assume that $a_1,\ldots, a_R\in (\Z/W_z\Z)^R$ satisfy 
\beq{eq:padrona}{v_p(a_i)<
2 \left[
\frac{z-3}{2R}
 \right], \, \forall p\leq z, \forall i\in \{1,\ldots,R\}.}
Let $x_1,\ldots,x_R\geq 1 $. The next lemma 
shows that for sufficiently general arithmetic progressions,
the sum over $\N^{R}$ can be written as a product of 
$R$ independent 
sums. 
This will make easier the application 
of results in analytic number theory 
on averages of arithmetic functions.
\begin{lemma}    \label{lem:splitting} 
For each $s\in \{1,-1\}$ the quantity  
$$
\# \left\{ 
\mathbf{m} \in \N^R
:  
\begin{array}{l}
m_i\leq x_i, \b m\equiv \b a \md {W_z} ,
(sm_1\cdots m_R, D)_\R=1, \\
(sm_1\cdots m_R, D)_{\Q_p}=1 \ \forall p\leq z, \\
2\mid v_p( m_i)  \ \forall p>z 
 \textrm{ with } (\frac{D}{p})=-1
 \textrm{ and } \forall i 
 \end{array}\right\}$$
 equals 
$$ 
\mathds 1((D,s)_\R=1 \, \mbox{ and } \,(D,sa_1\cdots a_R)_{\Q_p}=1,\,\forall p\leq z)
\prod_{i=1}^R F_i\l(\frac{x_i}{\gcd(a_i,W_z)}\r),$$ 
where   $F_i(T)$   is defined through 
$$\#\left\{
 1\leq c  \leq T: 
c \equiv \frac{a_i}{\gcd(a_i,W_z)} \md{\frac{W_z}{\gcd(a_i,W_z)}}, 
2\mid v_p( c)  \ \forall p>z 
 \textrm{ with } \l(\frac{D}{p}\r)=-1
\right\}.$$
\end{lemma}

\begin{remark} 
One can interpret this result as partitioning the counting function using a torsor and its twists. We will see that in our application we get a contribution from infinitely many twists.
\end{remark}

\begin{proof} 
Factor $m_i=c_id_i$, where   
$c_i,d_i \in \N$ are such that 
$\gcd(c_i,W_z)=1$ and  
all     prime factors of $d_i$ are in $[2,z]$.
We get 
 $c_i d_i=m_i \equiv a_i\md{W_z}$ so that 
 $d_i=\gcd(a_i,W_z)$
 due to $v_p(a_i)<z$ that comes from~\eqref{eq:padrona}.
We obtain the expression 
$$ \mathds 1((s, D)_\R=1)
\# \left\{ 
\mathbf{c} \in \N^{n+1}
:  
\begin{array}{l}
c_i\leq x_i/\gcd(a_i,W_z), 
c_i \equiv a_i/\gcd(a_i,W_z) \md{W_z/\gcd(a_i,W_z)},
\\ 
(s \gcd(a_1,W_z) \cdots \gcd(a_R,W_z) 
c_1 \cdots c_R, D)_{\Q_p}=1 \ \forall p\leq z, \\
2\mid v_p( c_i)  \ \forall p 
 \textrm{ with } (\frac{D}{p})=-1
 \textrm{ and } \forall i 
 \end{array}\right\}
.$$  
We omitted recording   
that $c_i$ is composed  only of primes $p>z$ as 
this is implied by~\eqref{eq:padrona} and 
$c_i \equiv a_i/\gcd(a_i,W_z) \md{W_z/\gcd(a_i,W_z)} $.
For primes $p\leq z$ 
 note that  $$v_p(
 s \gcd(a_1,W_z) \cdots \gcd(a_R,W_z) 
c_1 \cdots c_R
)\leq 
v_p( a_1\cdots a_R)
< 
2R \left[
\frac{z-3}{2R}
 \right]\leq z-3 =v_p(W_z)-3,$$
 hence,  \beq{eq:viv}{v_p(s \gcd(a_1,W_z) \cdots \gcd(a_R,W_z) 
c_1 \cdots c_R)=
v_p(  sa_1\cdots a_R).}
In addition, putting together 
$v_p(a_i) < v_p(W_z)$
and 
$c_i\equiv a_i/\gcd(a_i,W_z) \md{W_z/\gcd(a_i,W_z)}$
shows that  
$c_i\equiv a_i/\gcd(a_i,W_z) \md{p}$, thus, 
\beq{eq:viv2}{
\frac{ s \gcd(a_1,W_z) \cdots \gcd(a_R,W_z) 
c_1 \cdots c_R}{\displaystyle p^{v_p( s \gcd(a_1,W_z) \cdots \gcd(a_R,W_z) 
c_1 \cdots c_R)}}
\equiv 
\frac{sa_1 \cdots a_R }
{\displaystyle p^{v_p(\gcd(sa_1  \cdots   a_R,W_z))}} \md{p}.
} 
Thus, 
using~\eqref{eq:viv} and \eqref{eq:viv2}
together with   Lemma~\ref{lem:detectorslemma} and  Remark~\ref{remIHP}
leads us to 
  $$(D, s m_1 \cdots m_R)_{\Q_p}=
  (D, s a_1\ldots a_R)_{\Q_p}$$ for all      
  $2<p\leq z$.
  We have shown 
$v_2(
 s \gcd(a_1,W_z) \cdots \gcd(a_R,W_z) 
c_1 \cdots c_R
)\leq  v_2(W_z)-4$, hence, 
$$
\frac{ s \gcd(a_1,W_z) \cdots \gcd(a_R,W_z) 
c_1 \cdots c_R}{\displaystyle2^{v_2( s \gcd(a_1,W_z) \cdots \gcd(a_R,W_z) 
c_1 \cdots c_R)}}
\equiv 
\frac{sa_1 \cdots a_R }
{\displaystyle2^{v_2(\gcd(sa_1  \cdots   a_R,W_z))}} \md{2^4}.$$
Thus, 
 Remark~\ref{remIHP}
shows that 
 $(D, s m_1 \cdots m_R)_{\Q_2}=
  (D, s a_1\cdots a_R)_{\Q_2}$. 
\end{proof}
 For any Dirichlet character $\chi $
 we define the multiplicative function
 $$ f_\chi (n)=\chi(n) 
\mathds 1\l(2\mid v_p( n)  \ \forall p >z
 \textrm{ with } \l(\frac{D}{p}\r)=-1\r).$$
 \begin{lemma} \label{lem:koukou}
 For all $z>200$ and $T$ with $W_z\leq (\log T)^{25}$
 we have 
 $$F_i(T)=
\frac{\gamma_0 T}{(\pi \log T)^{1/2} }
\l(1+\l(\frac{D}{
a_i/\gcd(a_i,W_z) 
 }\r)\r)
\frac{\gcd(a_i,W_z) }{W_z} \prod_{p\leq z} \l(1-\frac{1}{p}\r)^{-1/2} 
+O\l (\frac{T}{\log T}\r ),$$ where
 the implied constant depends at most on $D$ and we define 
  $$ \gamma_0 =  \prod_{p> z}    \l(1-\frac{1}{p} \r)^{1/2}
\times \begin{cases}  
\l(1-  p^{-1} \r)^{-1}, & \textrm{ if } \, \, (\frac{D}{p})=1, \\
\l(1- p^{-2}  \r)^{-1}, & \textrm{ if }\, \, (\frac{D}{p})=-1.
\end{cases}$$ 
 \end{lemma}
 \begin{proof}
 Denoting $a:= 
a_i/\gcd(a_i,W_z), q:= W_z/\gcd(a_i,W_z)$, we  have   $\gcd(a,q)=1$. Hence, 
 \begin{equation}
 \label{eq:livietta}
 F_i(T) =\frac{1}{\phi(q)}\sum_{\chi\md q } \overline{\chi(a)}
 \sum_{ n\leq T  }f_\chi(n)  .
 \end{equation}
  To 
  estimate the sum over $n$ we shall
  employ~\cite[Theorem 13.8]{koukou}. 
  To
  verify~\cite[Equation (13.1)]{koukou}
  we let $Q=\exp(q^{1/100})$ so that if $A>0$ is fixed 
 and $x\geq Q$
  the Siegel--Walfisz theorem gives 
\begin{align*}
  \sum_{p\leq x } f_\chi(p)\log p=& 
    \sum_{\substack{ p\leq x\\ p>z\Rightarrow
    \left(\frac{D}{p}\right)=1} } \chi(p)\log p
=O(z)+
    \sum_{\substack{ p\leq x\\ 
    \left(\frac{D}{p}\right)=1} } \chi(p)\log p
    \\ = & 
  \frac{x}{2} \mathds 1 \l (\chi=\chi_0 \text{ or } 
  \chi\cdot  \left(\frac{D}{\cdot}\right)=\chi_0 \r )
  +O_A\l(z+\frac{x}{(\log x)^A}\r) ,\end{align*}
  where $\chi_0$ is the principal Dirichlet character
  modulo $q$.
The assumption $z> 200$ shows  that 
$$ \log \log z\leq (\log 0.9)+ \frac{z}{100} \Rightarrow
\log z \leq 0.9 \prod_{p\leq z}p^{1/100}
 \leq 0.9 q^{1/100}
 \Rightarrow z \leq Q^{0.9}\leq x^{0.9}
 ,$$ thus, the previous error term is 
 $O_A(x/(\log x)^A )$.  Now observe that 
  $4D$ divides $q$ due  
  to~\eqref{eq:padrona}, hence, 
  the character   $\chi(n) (\frac{D}{n})$ is
  principal if and only if $\chi(n)$ is induced by $(\frac{D}{n})$. 
Hence, 
by~\cite[Equations (13.09),(13.11)]{koukou} with
$k=J=\epsilon=1,\kappa=1/2$
we infer that when 
$T\geq (\log Q)^2=q^{1/50} $ one has 
 \begin{equation}
 \label{eq:livietta2} \sum_{n\leq T } f_\chi(n) = 
\frac{c(\chi)}{\sqrt \pi } 
\frac{T}{(\log T)^{1/2}}+O\l(\frac{T q^{1/50} }{(\log T)^{3/2}}\r),\end{equation}  where  $$ c(\chi)= 
\mathds 1 \l (\chi=\chi_0 \text{ or } 
  \chi\cdot  \left(\frac{D}{\cdot}\right)=\chi_0 \r )
  \prod_{p=2} ^\infty \l(\sum_{j=0} ^\infty \frac{f_\chi(p^j)}{p^{j}}\r)
   \l(1-\frac{1}{p} \r)^{1/2}$$
   and the implied constant 
  depends only on $D$. The assumption 
  $W_z \leq (\log T)^{25}$ shows $T\geq q^{1/50}$.
  Furthermore, it guarantees that 
  $q^{1/50} \leq W_z^{1/50}\leq (\log T)^{1/2}$, which 
  proves the error term in our lemma. 
  Regarding the   constant $c(\chi)$, note that 
  if $(\frac{D}{p})=-1$ and $p>z $
  then 
  $$\sum_{j=0} ^\infty \frac{f_\chi(p^j)}{p^{j}}=
  \sum_{t=0} ^\infty \frac{\chi(p^{2t})}{p^{2t}}=\l(1-\chi(p)^2 p^{-2}\r)^{-1}.$$ 
Similarly,  if $(\frac{D}{p})=1$ and $p>z$
  then  $\sum_{j\geq 0}  f_\chi(p^j) p^{-j}
=(1-\chi(p) p^{-1})^{-1}.$ 
For primes $p\leq z$
we have $p\mid q$, hence, $\chi(p^j)$ vanishes.
Putting everything together yields
$$c(\chi)=  
\prod_{p\leq z}    \l(1-\frac{1}{p} \r)^{1/2}
\prod_{p> z}    \l(1-\frac{1}{p} \r)^{1/2}
\cdot \begin{cases}  
\l(1-  \chi(p) p^{-1} \r)^{-1}, & (\frac{D}{p})=1, \\
\l(1- \chi(p)^2 p^{-2}  \r)^{-1}, & (\frac{D}{p})=-1.
\end{cases}
$$
In particular, $c(\chi_0)=c((\frac{D}{\cdot}))$
because 
if $\chi=(\frac{D}{\cdot})$ and 
$(\frac{D}{p})=1$ 
then $\chi(p)=1$, while, 
$\chi(p)^2=1$ for  all   $p>z$.
Hence, 
$$c(\chi_0)=  c\l(\l(\frac{D}{\cdot}\r)\r)=
\gamma_0  \cdot 
\prod_{p\leq z}    \l(1-\frac{1}{p} \r)^{1/2} 
,$$  where $\gamma_0 $ is as in the statement of this lemma.
Note that   $$\frac{1}{\phi(q)}=
\frac{1}{q} \prod_{p\leq z} \l(1-\frac{1}{p}\r)^{-1} =\frac{\gcd(a_i,W_z) }{W_z} \prod_{p\leq z} \l(1-\frac{1}{p}\r)^{-1} $$
  since $q$ is divisible by all primes $p\leq z$ owing 
  to~\eqref{eq:padrona}. We may hence write 
$$ \frac{ c(\chi_0)}{\phi(q)}= \gamma_0 
\frac{\gcd(a_i,W_z) }{W_z} \prod_{p\leq z} \l(1-\frac{1}{p}\r)^{-1/2} .$$ 
Serving this together with~\eqref{eq:livietta2} to~\eqref{eq:livietta}
yields 
 $$
 F_i(T) = \gamma_0   \frac{\gcd(a_i,W_z) }{W_z} 
 \prod_{p\leq z} \l(1-\frac{1}{p}\r)^{-1/2} 
 \frac{T}{(\pi\log T)^{1/2}}
\sum_{\substack{ \chi\md q \\ c(\chi)\neq 0 } }\overline{\chi(a)}
+O\l(\frac{T} {\log T}\r) .
 $$ Recall that 
 $c(\chi) \neq 0$ exactly when 
 $\chi$ is 
   principal or induced by 
 the quadratic character $(\frac{D}{\cdot})$, thus, 
 $$\sum_{\substack{ \chi\md q \\ c(\chi)\neq 0 } }\overline{\chi(a)}=  1+ \l(\frac{D}{a} \r),$$
 which is sufficient.     \end{proof}
\begin{remark}
The presence of      $1+\l(\frac{D}{
a_i/\gcd(a_i,W_z) 
 }\r)$ in Lemma~\ref{lem:koukou} 
 is natural. 
 Indeed, if there  is any integer $c$
 counted by $F_i$ then  its prime factorisation is 
 $c=p_1^{\alpha_1}\cdots p_r^{\alpha_r}$, 
 where the $p_i$ are distinct and such that 
 $\alpha_i$ is even whenever $(\frac{D}{p_i})=-1$.
 Then  $$1=\prod_{i=1}^r \l(\frac{D}{p_i^{\alpha_i} }\r)=
 \l(\frac{D}{c}\r)=\l(\frac{D}{a_i/\gcd(a_i,W_z)}\r),$$
where the last equality comes from the congruence 
 $ c \equiv a_i/\gcd(a_i,W_z) \md{W_z/\gcd(a_i,W_z)}$ 
 and  the fact that $4D$ divides $W_z/\gcd(a_i,W_z)$.
 \end{remark} 
 
 \begin{lemma} \label{lem:adrianosyria}
  Assume that $z>200$, 
  $W_z (\log z)^{R/2}\leq (\log \min x_i )^{1/4}$ and that~\eqref{eq:padrona} holds.
  For each   $s\in \{1,-1\}$ we have 
$$ \# \left\{ \mathbf{m} \in \N^R:  
\begin{array}{l} \b m\equiv \b a \md {W_z} ,
(sm_1\cdots m_R, D)_\R=1, \\
(sm_1\cdots m_R, D)_{\Q_p}=1 \ \forall p\leq z, \,\,
m_i\leq x_i \ \forall i, \\
2\mid v_p( m_i)  \ \forall p>z 
 \textrm{ with } (\frac{D}{p})=-1
 \textrm{ and } \forall i  \end{array}\right\}
 =\frac{ \c M+O(\c E) }{ W_z^{R}}
 \prod_{i=1}^R
 \frac{x_i} {  (\log x_i)^{1/2} }
 ,$$   where 
   $$ \c M=
 \mathds 1(\eqref{eq:pass the joint equidistribution})
 \l(\frac{ 2\gamma_0 }{\sqrt \pi}\r)^R
\prod_{p\leq z}\l(1-\frac{1}{p} \r)^{-R/2} 
, \ \c E=  \frac{1}{  (\log \min x_i )^{1/4}}
 ,$$   the implied constant depends only on $D$ and $R$
and  \beq{eq:pass the joint equidistribution}{
(D,s)_\R=1, \quad (D,sa_1\cdots a_R)_{\Q_p}=1
\  \forall p\leq z, \quad 
\l(\frac{D}{a_i/\gcd(a_i,W_z)}\r)=1 \ 
\forall i .}   \end{lemma}
 \begin{proof}The assumption 
  $W_z\leq (\log (x_i/\gcd(a_i,W_z)))^{25}$
  of Lemma~\ref{lem:koukou} is met  due to the 
 condition   $ W_z (\log z)^R
 \leq (\log \min x_i)^{1/4}$ . We may thus feed 
  Lemma~\ref{lem:koukou}
  into Lemma~\ref{lem:splitting}. The   error term is 
 \begin{align*}
      \ll  & \sum_{i=1}^R 
 \frac{x_i}{ \log x_i}
\l(
\prod_{j\neq i }  \frac{x_j}{ W_z 
 (\log x_j)^{1/2}} \prod_{p\leq z} 
 \l(1-\frac{1}{p}\r)^{-1/2}
\r)
\\ \ll &   \frac{x_1\cdots x_R}
 { W_z^R (\log x_1)^{1/2} \cdots  (\log x_R)^{1/2}}
 \frac{ W_z (\log z)^{R/2}}{(\log \min x_i )^{1/2}}.
 \end{align*} This is satisfactory due to the assumption 
$W_z (\log z)^R\leq (\log \min x_i )^{1/4}$.

 The ensuing main term equals the one in our lemma save for the fact that each $\log x_i$ is replaced by 
 $\log(x_i/\gcd(a_i,W_z))$. Our assumptions
 imply that  $W_z  \leq (\log  x_i )^{1/4}$, hence, 
 $$\frac{1}{\sqrt{\log (x_i/\gcd(a_i,W_z))}}=
\frac{1}{\sqrt {\log x_i}} 
\l(1+O\l(\frac{\log W_z}{\log x_i }\r)\r)=
\frac{1}{\sqrt {\log x_i}} 
\l(1+O\l(\frac{\log \log x_i}{\log x_i }\r)\r),$$
 thus, introducing the error term $$\ll
 \frac{\log \log \min x_i }{\log \min x_i }
 \frac{(\log z)^{R/2}}{   W_z^R } 
\prod_{i=1}^R  \frac{x_i}{\sqrt{\log x_i}}
,$$ which is satisfactory.
 \end{proof} 
 \begin{remark}
 \label{rem:immolation}It is not difficult to 
 see that  for any fixed $A>0$ one has 
 $$\gamma_0  =1+O_A\l(\frac{1}{ (\log z)^A}\r) .$$ 
Indeed, we have 
$$\log \gamma_0 =O\l(\frac{1}{z}\r)+
\sum_{p>z} \l(-\frac{1}{2p}+
\frac{\mathds 1((\frac{D}{p})=1)}{p}\r)
.$$ 
 Using the strong version of the prime number theorem 
 giving arbitrary logarithmic saving 
 and partial summation one can prove that
      $$\sum_{ z<p\leq t }\frac{1}{2p}=\frac{1}{2}
 \log \frac{\log t}{\log z} +O_A\l(\frac{1}{(\log z)^A}\r )
 \textrm{ and }
 \sum_{\substack{z<p\leq t \\
 (\frac{D}{p})=1}}\frac{1}{p}=\frac{1}{2}
 \log \frac{\log t}{\log z}
 +O_A\l(\frac{1}{(\log z)^A}\r ).$$  
This leads to $\log \gamma_0 \ll_A (\log z)^{-A}$,
hence   $\gamma_0 =1+O_A((\log z)^{-A})$.
 \end{remark}
\subsection{Analytic lemmas}
\label{s:arcturus is a most fucking awsome band}
We first  need a version of the large sieve
as given in the work of Serre~\cite{serresieve}
but with boxes of arbitrary size.
It can be proved by combining work of 
Huxley and Kowalski; here 
  we recall the version 
that is stated in~\cite[Lemma 5.1]{wilson}.
\begin{lemma}
\label{lem:shiuconseqnjdf8d8}
Fix $s,R \in \mathbb N$ and for any prime $p$ let  
$Y_p$ be a subset of $(\Z/p^s\Z)^R$. Then 
for all $N_1,\ldots, N_R ,L  \geq 1$ we have 
$$
\#\{\b n \in \Z^R: |n_i| \leq N_i \ \forall i, \,\, 
\b n \md{p^s} \notin Y_p \ \forall p\leq L\}
\ll_{s,R} \frac{\prod_{i=1}^R (N_i+L^{2s} )}{ F(L)}
$$ where the implied constant depends at most on $s,R$
and we define 
$$F(L):=\sum_{\substack{m \ \mathrm{square}-\mathrm{free} 
\\  1\leq m\leq L }} 
\
\prod_{\substack{p \ \mathrm{ prime}\\
p\mid m}} \frac{\#Y_p}{p^{Rs}-\#Y_p } .$$
\end{lemma}
We provide a direct application below. 
Recall the definition of $\c N_D$ in~\eqref{def:norms}.
\begin{lemma}
\label{lem:shiuconsapplicsaser}
Fix $R \in \mathbb N$. Then 
for all $N_1,\ldots, N_R   > 1$ we have 
$$
\#\{\b n \in \Z^R:
n_1\cdots n_R \in \c N_D,
|n_i| \leq N_i \ \forall i\}
\ll_{D,R}   \frac{N_1 \cdots N_R}{ 
\min_i (\log  N_i)^{R/2} }
$$ where the implied constant depends at most on $D$ and 
$R$. \end{lemma} \begin{proof}  We employ 
Lemma~\ref{lem:shiuconseqnjdf8d8}
with  $s=2,R=r, L=\min \{N_i\}$ and 
$$Y_p=\{\b y \in (\Z/p^2\Z)^R: 
  y_1\cdots y_R\equiv 0\md{p}
  \text{ and }    
  y_1\cdots y_R\not\equiv 0\md{p^2}
  \}$$ for primes $p$
  such that $  (\frac{D}{p})=-1  $ and 
     $ p>2R  $. We let $Y_p =\emptyset $ for all other primes. 
Note that if an integer $m$ is in $\c N_D$
then for every prime $p$ with $(\frac{D}{p})=-1$
we must have $v_p(m) \neq 1$. Therefore, 
for all primes $p$ with $(\frac{D}{p})=-1$
and all $\b n$ with $n_1\cdots n_R$ in $\c N_D$
we must have that the image of 
$n_1\cdots n_R$ is in the 
complement of $Y_p \md{p^2}$.
From Lemma~\ref{lem:shiuconseqnjdf8d8}
we obtain the upper bound 
 $ \ll  N_1\cdots N_R / F(\min N_i).$

 To estimate $F$ we need a lower bound on $\#Y_p$. 
 Note $\#Y_p=M_p-E_p$, where 
 $$M_p=\#\{\b y \in (\Z/p^2\Z)^R: \exists i \in \{1,\dots,R\}, \, p\mid y_i \},
E_p= \#\{\b y \in (\Z/p^2\Z)^R: p^2\mid y_1\cdots y_R \}.$$
We clearly have 
 $M_p=R p^{2(R-1)} 
\#\{ y_1 \in \Z/p^2\Z: p\mid y_1  \}=R p^{2R-1} $ 
and 
$$E_p\leq \sum_{i\neq j }
\#\{\b y \in (\Z/p^2\Z)^R: p\mid (y_1,y_2) \}
+ R
\#\{\b y \in (\Z/p^2\Z)^R: p^2\mid y_1  \}
\leq 2R^2 p^{2R-2}.$$
Hence, $ \#Y_p \geq p^{2R} (R/p-
2R^2/p^2 )$, thus,
$$ \frac{\#Y_p}{p^{2R}-\#Y_p } 
\geq 
 \frac{\#Y_p}{p^{2R} } 
 \geq \frac{f(p)}{p},
 $$
 where $f$ is a multiplicative function defined by 
 $$
 f(p^e)= \mathds 1\l(e=1
 \ \mbox{and} \ 
\Big(\frac{D}{p}\Big)=-1
\ \mbox{and} \  p>2R \r) 
R\l(1-\frac{2R}{p} \r)
$$  We obtain   
 $F(T) \geq  \sum_{ m\leq T  }  f(m) /m$. 
 By~\cite[Theorem 13.2]{koukou} with $\kappa=R/2$  
 and $J=1$ we get 
 $$ \sum_{m\leq x} f(m) \gg x(\log x)^{R/2-1}.$$
 By partial summation   we then  deduce 
$ F(T) \gg (\log T)^{R/2}$, which concludes the proof.    \end{proof} 
By the change of variables 
$n_1=p^{2\alpha} n_1'$ and  Lemma~\ref{lem:shiuconsapplicsaser}
we obtain the following result.

\begin{lemma}
\label{lem:shiucuwer6}
For all 
$\alpha, R \in \mathbb N,
N_1,\ldots, N_R   > 1$ 
and prime powers $p^{2\alpha }\leq \sqrt N_1$ we have 
$$
\#\{\b n \in \Z^R: 
n_1\cdots n_R \in \c N_D,
|n_i| \leq N_i \ \forall i, n_1\equiv 0 \md{p^{2\alpha}}\}
\ll  \frac{N_1 \cdots N_R}{ p^{2\alpha }
\min_i (\log  N_i)^{R/2} }
$$ where the implied constant depends at most on $D$ and 
$R$.
\end{lemma} 

\begin{lemma}
    \label{lem:freezers} For  any prime $p$, $m\in \mathbb N$ and any  
    $\boldsymbol \nu   \in \Z^R$ 
    the limit 
    $$ \sigma_p(\boldsymbol \nu )=\lim_{m\to\infty }
    \frac{\#\{ \b t \in (\Z/p^m\Z)^{n+1}:
    (f_1(\b t),\ldots, f_R(\b t ))=\boldsymbol \nu  \}}{p^{m
    (n+1-R)}}$$ exists. Furthermore, 
    there exists $c=c(\b f )>0$ such that 
\beq{eq:freezers1}{   \sigma_p(\boldsymbol \nu )= 1+O(p^{-1-c})}  and \beq{eq:freezers2}{ 
 \frac{\#\{ \b t \in (\Z/p^m\Z)^{n+1}: \b f (\b t) 
 =\boldsymbol \nu \}}{p^{m(n+1-R)}}
 =\sigma_p(\boldsymbol \nu ) +O(p^{-(m+1)(c+1)}) ,} where the implied constants depend only on $\b f$.
    \end{lemma}
    \begin{proof} By the last displayed equation 
    in~\cite[page 259]{MR0150129} the fraction in the limit equals  
    $$ 
    \sum_{r=0}^m p^{-r(n+1) }
    \sum_{\substack{ \b a \in (\Z/p^r\Z)^R
    \\p\nmid \b a }}
    \mathrm e \l(-p^{-r} \sum_i a_i \nu_i\r)
    \sum_{\b x \in (\Z/p^r\Z)^{n+1} }
    \mathrm e \l(p^{-r} \sum_i a_i f_i(\b x )\r).
    $$  The bound~\eqref{eq:singseries}
    shows that the sum over $\b x $ is $\ll p^{r(n-R-c)}$. 
    Therefore, the overall expression is 
    $$\ll\sum_{r=0}^m p^{r(-1-c)},$$ which therefore converges 
    absolutely as $m\to \infty$. This proves that the limit converges. 
The difference has modulus 
$  \ll  p^{-(m+1) (c+1)}$, hence proving~\eqref{eq:freezers2}. 
For $m=0$ it  proves~\eqref{eq:freezers1}.   \end{proof}

For a prime $p$ let $\mu_p$ 
denote the standard $p$-adic Haar measure.
\begin{lemma} 
\label{lem:mitsotakigamiesai2}
For any prime $p$, any $i=1,\ldots, R$
and any $k\in \N$ we have 
$$\mu_p\{\b x \in \Z_p^{n+1} \colon p^k \mid f_i(\b x ) \} \ll p^{-k} 
,$$ where the implied constant 
depends only on $f_i$. 
\end{lemma}
\begin{proof}
We have 
\begin{align*}
\mu_p\{\b x \in \Z_p^{n+1} \colon p^k \mid f_i(\b x ) \}&=
\frac{\#\{\b x \in (\Z/p^k \Z)^{n+1}: f_i(\b x) =0 \}}{p^{k(n+1)}}\\
&=p^{-kR}
\sum_{\substack{ \boldsymbol \nu \in (\Z/p^k\Z)^R \\  \nu_i=0 } }
\frac{\#\{\b x\in (\Z/p\Z)^{n+1} \colon 
(f_1(\b x ),\ldots, f_R(\b x) )=\boldsymbol \nu \}}{p^{k(n+1-R)}}.\end{align*}
By~\eqref{eq:freezers2} 
this becomes
$$   O\left(p^{-(k+1)(c+1)}\right) +p^{-kR}
\sum_{\substack{ \boldsymbol \nu \in (\Z/p^k\Z)^R \\  \nu_i=0 } }\sigma_p(\boldsymbol \nu ) $$ 
and then by~\eqref{eq:freezers1}
this is
\[
O\left(p^{-(k+1)(c+1)}\right)  +p^{-kR}
\sum_{\substack{ \boldsymbol \nu \in (\Z/p^k\Z)^R \\  \nu_i=0 } }1 =
O\left(p^{-(k+1)(c+1)}\right) 
 +p^{-k}
 \ll p^{-k}. \qedhere
\]
\end{proof}

\section{Proof of the asymptotic}
\label{s:proofproofproof}
We now give the proof of Theorem~\ref{thm:mainthrm} by applying the circle method result 
in Theorem~\ref{thm:vachms}. Before doing so, we show in \S\ref{s:preparegrounds}
that the values of the polynomials $f_i(\b x )$
lie the  special arithmetic progressions that satisfy the assumption~\eqref{eq:padrona}
of Lemma~\ref{lem:adrianosyria}. We then prove Theorem~\ref{thm:mainthrm} 
in \S\ref{s:activation} subject to Lemma~\ref{lem:ozzy}  which states that 
the entities $\gamma(\b s )$ appearing in the leading constant are well defined.
The proof of 
Lemma~\ref{lem:ozzy} is not straightforward;
it is given in the separate section~\ref{s:prfoz}.

\subsection{Preparing the grounds} \label{s:preparegrounds}
Recall the definition of $N(B)$ in~\eqref{def:countingfunction}.
We can write  
  \[ N(B)=\frac{1}{2}
\# \left\{  \b x  \in 
\left(\Z\cap \left[-B^{\frac{1}{n+1}},B^{\frac{1}{n+1}}\right]\right)^{n+1}
\setminus\{\b 0\}:
\gcd(x_0,\ldots,x_n)=1, 
f_1(\b x) \cdots f_R(\b x) \in \mathcal{N}_D\right\},\] where 
 $\c N_D$  is  as in~\eqref{def:norms}.
By M\"obius inversion this   becomes 
$  \frac{1}{2} \sum_{k\geq 1 }\mu(k)
N_0(B^{1/(n+1)}/k)$,where 
$$N_0(P)=\# \left\{ 
\b x \in (\Z\cap [-P,P])^{n+1}\setminus\{\b 0 \}:
\begin{array}{l}
(D, f_1(\b x ) \cdots f_R(\b x ) )_\R=1,\\
(D, f_1(\b x ) \cdots f_R(\b x ) )_{\Q_p}=1
\ \forall p\,\,   \rm{ prime}
\end{array}
\right\}.$$
The bound
$N_0(P) \ll P^{n+1}$ shows that the 
contribution of $k>(\log B)^{2R/n}$
is $$ \ll \sum_{k>(\log B)^{2R/n} }
\frac{B }{k^{n+1}} \ll \frac{B }{(\log B)^{2R}}
\leq   \frac{B }{(\log B)^{1+R/2}},$$ hence, 
 \beq{eq:reni}
{N( B)=\frac{1}{2} 
\sum_{1\leq k \leq (\log B)^{2R/n}  }\mu(k)
N_0\l(\frac{B^{\frac{1}{n+1}}}{k} \r)
+O\l( \frac{B}{(\log B)^{1+R/2}}\r).} 
As this point we should use  
the circle method tool, namely
Theorem~\ref{thm:vachms}. However, 
since this only applies to 
specific arithmetic progressions,
we first show that the remaining 
progressions can be ignored 
up to a negligible error term.

We begin by showing 
that for 
most $\b x $ in $N_0(P)$
the integers 
$f_1(\b x ),\ldots, f_R(\b x )$ 
are relatively coprime 
with respect to large primes.

\begin{lemma}
\label{lem:codimY}
Fix any $i\neq j $ among the integers in $\{1,\dots,R\}$.
Then  for all $P, z>1$  we have $$\# \left\{ 
\mathbf{x} \in (\Z\cap [-P,P])^{n+1}
:  \begin{array}{l}
f_1(\b x) \cdots f_R(\b x) \in \c N_D, 
\\ \exists   p>z :
 p\mid (f_i(x), f_j(x) )
\end{array} \right\}\ll 
\frac{1}{z (\log z)}  
\frac{P^{n+1}}{(\log P)^{R/2}}
+\frac{P^{n+1/2}}{(\log P)^{1/2}}
,$$ where the implied constant 
only depends on $D,R$ and the $f_i$.
\end{lemma}\begin{proof} 
By Lemma~\ref{corl:uperbnd}
with   $\c A=\{\b m\in \Z^{R}: 
m_1\cdots m_R \in \c N_D \}$ 
we obtain the  bound  
$$ \ll P^{n+1-d}+
P^{n+1-Rd} \sum_{1\leq i<j \leq R}
\sum_{p>z}\# \left\{ 
\mathbf{m} \in (\Z\setminus\{0\})^{R}
:  \begin{array}{l}
m_1 \cdots m_R \in \c  N_D,\
|m_i| \leq b P^d,\\
 p\mid (m_i, m_j )
\end{array} \right\},$$
where  $P^{n+1-d}$ 
comes from the  cases where 
$f_1(\b x ) \cdots f_R(\b x )=0$.
The contribution of $p\geq \sqrt P$ 
is     $$ \ll P^{n+1-Rd} \sum_{p\geq \sqrt P } 
P^{d(R-2)} \frac{P^{2d} }{p^2} \ll \frac{ P^{n+1/2}}{\log P} .$$ Here we used the bound $O(P^{2d}/p^2)$,
rather than $O(1+P^{2d}/p^2)$, as we have excluded 
the cases where $m_im_j=0$.
If $z\geq \sqrt P$ there are no more primes in the
sum and the bounds so far are satisfactory.
If $z<\sqrt P$ then  for the primes $p\in (z,\sqrt P]$ 
 a change of variables gives  
$$ \# \left\{ 
\mathbf{m} \in (\Z\setminus\{0\})^{R}
:  \begin{array}{l}
m_1 \cdots m_R \in \c  N_D,\\
|m_i| \leq b P^d, 
p\mid (m_i, m_j )
\end{array}\right\} 
\leq  \# \left\{ 
\mathbf{m} \in \Z^{R}
:    m_1 \cdots m_R \in \c N_D,
|m_\ell | \leq P_\ell \  \forall \ell 
\right\} ,$$ where 
$P_\ell = b P^d$ for $\ell\notin\{i,j\}$
and  $P_i=P_j=b P^d/p$. 
By Lemma~\ref{lem:shiuconsapplicsaser}
with $N_\ell=P_\ell$
we get   $$
\# \left\{ 
\mathbf{m} \in \Z^{R}
:    m_1 \cdots m_R \in \c N_D,
|m_\ell | \leq P_\ell \  \forall \ell 
\right\} \ll \frac{P^{dR}}{p^2(\log P)^{R/2}}$$
because the inequality $p\leq \sqrt P$
ensures that 
$\min \log P_\ell \gg \log P$.
This concludes the proof
by noting that $\sum_{p>z}p^{-2} \ll 1/(z\log z)$.
\end{proof}

We next show  
that for 
most $\b x $   in $N_0(P)$
the integers 
$f_i(\b x )$
are not divisible by 
a large prime  power.
\begin{lemma}
\label{lem:vinci_Gismondo}
For all $P>1$, all positive 
  integers $\alpha$
and   primes $p$ with $p^{2\alpha}\leq P^{1/2}$
we have $$
\# \left\{ 
\mathbf{x} \in (\Z\cap [-P,P])^{n+1}
:  
\begin{array}{l}
f_1(\b x) \cdots f_R(\b x) \in \c N_D,\\
\exists i \colon
f_i(\b x ) \equiv 0 \md{p^{2\alpha}}
\end{array}
\right\}\ll \frac{1}{p^{2\alpha}} 
\frac{P^{n+1}}{(\log P)^{R/2}},$$ where the implied constant 
only depends on $D,R$ and the $f_i$.  
\end{lemma}\begin{proof} 
By Lemma~\ref{corl:uperbnd}
for a suitable     $b=b(f_i)$ 
and 
 $\c A=\{\b m\in \Z^{R}: m_1\cdots m_R \in \c N_D \}$
we get  $$ \ll P^n+ P^{n+1-Rd}    \# \left\{ 
\mathbf{m} \in (\Z\setminus\{0\})^{R}
:   \begin{array}{l}
m_1 \cdots m_R \in \c  N_D,
\\ p^{2\alpha}\mid m_i, |m_i| \leq b P^d,
\end{array} \right\},$$
where   $P^n$ comes from the 
cases where $f_1(\b x) \cdots f_R(\b x )=0$.
The proof concludes by using Lemma~\ref{lem:shiucuwer6}.
\end{proof}
Putting together the bounds
we come to the following conclusion
regarding $N_0(P)$.
\begin{lemma}
\label{lem:codijhsdcontentomY}
For all $P,z>1$ and 
 positive integers $\alpha$
  satisfying 
 $z^{2\alpha} \leq P^{1/2}$,  
we have  
\begin{align*}
N_0(P)&= 
\# \left\{ 
\b x \in \Z^{n+1} :
\begin{array}{l} 
(D, f_1(\b x ) \cdots f_R(\b x ) )_\R=1,0<\max|x_i| \leq P\\
(D, f_1(\b x ) \cdots f_R(\b x ) )_{\Q_p}=1
\ \forall p\leq z, \\
2\mid  v_p(f_i(x)   )  \ \forall p>z 
 \textrm{ with } (\frac{D}{p})=-1
 \textrm{ and } \forall i,\\
   v_p(f_i(x)   ) < 2\alpha   \ \forall p\leq z 
\end{array}
\right\} \\
&+O\left (
\frac{1}
{ \min \{z\log z, 4^\alpha,
\sqrt {P/(\log P)^R}
\}}
\cdot \frac{ P^{n+1}}{(\log P)^{R/2}}
\right),
\end{align*}
where the implied constant only depends on $D$ and the $f_i$.
\end{lemma}

\begin{proof} 
By Lemma~\ref{lem:vinci_Gismondo}
the terms $\b x $ failing 
 $v_p(f_i(x)   ) < 2\alpha $ for all
$p\leq z $  have cardinality 
$$\ll \frac{P^{n+1} }{(\log P)^{R/2}} \sum_{p\leq z} p^{-2\alpha}
\ll
\frac{P^{n+1} }{4^\alpha (\log P)^{R/2}} 
$$ by~\eqref{eq:adriano in siris by pergolesi}. By Lemma~\ref{lem:codimY} we can restrict attention 
to 
those $\b x $ in $N_0(P)$
for which there is no prime $p>z$
dividing both $f_i(\b x), f_j(\b x )$
for some $i\neq j$ as long as we introduce 
a negligible error term. 
Hence, for the remaining $\b x $
and for   $p>z$ with $(\frac{D}{p})=-1$
we have 
$(D, f_1(\b x ) \cdots f_R(\b x ) )_{\Q_p}=1$
equivalently when 
$2\mid v_p(f_1(\b x ) \cdots f_R(\b x ))$,
which is  
equivalent to 
$2\mid v_p(f_i(\b x ) )$ for all $i$.
 Thus, $N_0(P)$ equals  
 $$ 
\# \left\{ 
\b x \in \Z^{n+1} :
\begin{array}{l} 
(D, f_1(\b x ) \cdots f_R(\b x ) )_\R=1,0<\max|x_i| \leq P
\\
(D, f_1(\b x ) \cdots f_R(\b x ) )_{\Q_p}=1
\ \forall p\leq z, \\
2\mid  v_p(f_i(x)   )  \ \forall p>z 
 \textrm{ with } (\frac{D}{p})=-1
 \textrm{ and } \forall i
 \\ p>z, i\neq j \Rightarrow p\nmid (f_i(\b x), f_j(\b x ) )
 ,\\    v_p(f_i(x)   ) < 2\alpha   \ \forall p\leq z 
\end{array}
\right\}
$$
 up to
 a negligible error term. 
Lastly, 
the condition $p\nmid (f_i(\b x), f_j(\b x ) )$
can be dispensed with 
by alluding  once again 
to  Lemma~\ref{lem:codimY}.
\end{proof}  

 \subsection{Application of  the circle method tool} \label{s:activation}
Write the right-hand side of 
Lemma~\ref{lem:codijhsdcontentomY}    as 
\beq{bwv532}{\sum_{\substack{ s_1,\ldots, s_R\in \{1,-1\} \\ (D, s_1 \cdots s_R)_\R=1 }}
\# \left\{ 
\b x \in \Z^{n+1} :
\begin{array}{l} 
s_if_i(\b x ) >0,
0<|x_i| \leq P\ \forall i,\\
(D, f_1(\b x ) \cdots f_R(\b x ) )_{\Q_p}=1
\ \forall p\leq z, \\
2\mid  v_p(f_i(x)   )  \ \forall p>z 
 \textrm{ with } (\frac{D}{p})=-1
 \textrm{ and } \forall i,\\
   v_p(f_i(x)   ) < 2\alpha   \ \forall p\leq z 
\end{array}
\right\}
.}Set $m_i=s_i f_i(\b x )$
so that the new cardinality equals 
$$
\sum_{\substack{ 
\b x \in \Z^{n+1} \cap [-P,P],
\\ \min_j s_jf_j(\b x ) >0 
}} k(s_1 f_1(\b x ),\ldots, s_R f_R(\b x )),
$$ where $k:\N^R\to \{0,1\}$ is defined as the indicator of the 
event 
$   v_p( m_i  ) < 2\alpha$ for all $p\leq z 
$ intersected with  
$$(D, s_1\cdots s_R
m_1 \cdots m_R
)_{\Q_p}=1
\ \forall p\leq z,  \quad  
2\mid  v_p( m_i   )  \ \forall p>z 
 \textrm{ with } \l (\frac{D}{p} \r )=-1
 \textrm{ and for all } i
.$$
We are now ready to employ Theorem~\ref{thm:vachms},
for which we take $z$ to be an integer function of $P$
that will be defined later and we set
$m_p(z)=z$ in Definition~\ref{def:wzez}
and  $\alpha =[(z-3)/(2R)]$. 
Let 
$$\rho(\b a, W_z) 
=
\mathds 1(\b a )
 \l(
\frac{2\gamma_0}{\sqrt \pi W_z}\r)^R
\prod_{p\leq z}\l(1-\frac{1}{p} \r)^{-R/2}
,$$ where $\mathds 1 $
is the indicator of 
$$ v_p( a_i  ) < 2\l [\frac{z-3}{2R}\r]  \textrm{ and }
(D,s_1\cdots s_R a_1\cdots a_R)_{\Q_p}=1
\  \forall p\leq z, \quad 
\l(\frac{D}{a_i/\gcd(a_i,W_z)}\r)=1 \ 
\forall i 
.$$By injecting
Lemma~\ref{lem:adrianosyria}
into 
Theorem~\ref{thm:vachms}
we infer that~\eqref{bwv532} equals 
$$
P^{n+1} 
 \l(
\frac{2\gamma_0 }{\sqrt \pi  }\r)^R
\prod_{p\leq z}\l(1-\frac{1}{p} \r)^{-R/2}
\sum_{\substack{ s_1,\ldots, s_R\in \{1,-1\} \\ (D, s_1 \cdots s_R)_\R=1 }}  H_\infty(\b s) 
\widehat{\gamma}_z(\b s )
 +\c R P^{n+1} 
, $$
 where 
 $$\c R\ll \frac{\|k\|_1}{P^{Rd}}
 (P^{-\delta} +\widetilde{\epsilon}(z)+z^{-c})
 + \frac{1  }{(\log P)^{1/4+R/2}   }
 , \quad    H_\infty(\b s)= 
 \hspace{-0.3cm}
 \int\limits_{\substack{
|\b t |\leq 1 \\ 
\min_j s_j f_j(\b t ) > P^{-d} } } 
\hspace{-0.3cm}
\prod_{i=1}^R \frac{1}{(\log P^d
(s_i  f_i(\b t ))^{1/2}}
\mathrm d \b t $$ and   
$$ \widehat{\gamma}_z(\b s ):=
 \frac{ 1}{W_z^{n+1}} 
\# \left\{ \b t \in (\Z/ W_z \Z )^{n+1} :
\begin{array}{l} 
 \displaystyle v_p( f_i(\b t )   ) < 2\left[\frac{z-3}{2R} \right]  
 \ \forall p\leq z , \forall i,  \\[2mm]
(D,  f_1(\b t ) \cdots f_R(\b t ) )_{\Q_p}=1 \ \forall p\leq z,  \\
\l(\frac{D}{s_i f_i(\b t)/\gcd(f_i(\b t) ,W_z)}\r)=1 \ 
\forall i \end{array} \right\}  .$$ 
The bound  $\widetilde{\epsilon}(z)\ll 4^{-z}$ can be 
proved as in~\eqref{eq:adriano in siris by pergolesi}, while,   $ \|k\|_1\ll P^{dR} (\log P)^{-R/2}$
follows from  Lemma~\ref{lem:shiuconsapplicsaser}.
Thus, Lemma~\ref{lem:adrianosyria} gives 
$$\c R\ll
\frac{ 
P^{-\delta} +4^{-z}+z^{-c}+ (\log P)^{-1/4} 
 }{(\log P)^{R/2}} $$ as long 
 as $W_z (\log z)^{R/2} \leq (\log P)^{1/4}$.
Note that  
$W_z\leq \mathrm e^{m(z) (z +z/2\log z)}$ since 
$$
\log W_z=z\sum_{p\leq z} \log p  
\leq  z \l(z+\frac{z}{ 2\log z}\r) $$ by~\cite[Equation (3.15)]{MR0137689}.
Hence,   for all large $z$ we have 
 $W_z(\log z)^R \leq 
 \mathrm e^{z (z +z/\log z)}$. Therefore, 
if $5z^2 \leq \log \log P$ we 
infer that $W_z(\log z)^R \leq (\log P)^{1/4}$.
We choose 
$z=\left[\sqrt{(\log \log P)/5}\right]$  so that 
$\c R\ll (\log P)^{-R/2} (\log \log P)^{-c/2}$.

Furthermore, $|\b t|\leq 1$ we see that  
$f_j(\b t)$ is bounded independently of $\b t$, 
thus
$$\frac{1}{\sqrt{ \log (P^d
(s_j  f_j(\b t )))}}
=
\frac{1}{\sqrt{d\log P}}
\frac{1}{ \sqrt{1+\frac{\log (s_j  f_j(\b t ))}{d\log P}} }
=
\frac{1}{\sqrt{d\log P}}
\left(1+O\l(\frac{1}{\log P}\r)\right).$$
In particular, 
 $$ H_\infty(\b s)= 
 \frac{1}{(d\log P)^{R/2}}
 \mathrm{vol}
 \l(|\b t |\leq 1: \min_j s_j f_j(\b t ) > P^{-d}\r)
\left(1+O\l(\frac{1}{\log P}\r)\right).$$
By~\cite[Lemma 1.19]{circle}
the volume of 
$|\b t| \leq 1$
for which 
$|f_j(\b t)| \leq P^{-d}$ is 
$\ll P^{-d}$, hence  $$ H_\infty(\b s)= 
 \frac{ \gamma_\infty(\b s )}{(d\log P)^{R/2}}
\left(1+O\l(\frac{1}{\log 
P}
\r)\right),$$ where $\gamma_\infty(\b s ):=\mathrm{vol}
 \l\{\b t \in [-1,1]^{n+1}:   \mathrm{sign}(f_j(\b t ))=s_j 
 \forall j\r\} $.
Putting all estimates  
 into Lemma~\ref{lem:codijhsdcontentomY}
 yields   
\beq{messiah nostrum}{
\frac{N_0(P)}{P^{n+1}}=  
 \l(
\frac{2\gamma_0 }{\sqrt {\pi d (\log P) }}\r)^R
\prod_{p\leq z}\l(1-\frac{1}{p} \r)^{-\frac{R}{2}}
\hspace{-0.2cm}
\sum_{\substack{ \b s\in \{1,-1\}^R \\ (D, s_1 \cdots s_R)_\R=1 }} 
\gamma_\infty(\b s )
\widehat{\gamma}_z(\b s )
+O\l(\frac{ (\log P)^{-\frac{R}{2} } }{(\log \log P)^{\frac{c}{2}}}
\r). } 
  \begin{lemma}  \label{lem:no more tears} We have  
  $   \widehat{\gamma}_z(\b s ) =     {\gamma}_z(\b s )
  +O(z2^{-z/R})$, where for $T\geq 1 $ we let 
    $$ 
 \gamma_T(\b s ):= \frac{ 1}{\prod_{p\leq T} p^{T(n+1)}} 
\# \left\{ \b t \in \l(\Z/ \prod_{p\leq T} p^T \Z \r)^{n+1} : \begin{array}{l} 
 v_p\l ( \prod_{i=1}^R f_i(\b t )   \r) < T  \ \forall 3 \leqslant p\leqslant T,   \\[2mm]
  v_2\l ( \prod_{i=1}^R f_i(\b t )   \r) < T    -3,  \\
(D,  f_1(\b t ) \cdots f_R(\b t ) )_{\Q_p}=1 \ \forall p\mid q,  \\
\l(\frac{D}{s_i f_i(\b t)/\gcd(f_i(\b t) ,\prod_{p\leq T}p^T )}\r)=1 \ 
\forall i
\end{array}
\right\} $$ and  the implied constant only depends on $f_i$ and $D$.
\end{lemma} \begin{proof}
  The condition $v_p( f_i(\b t )   ) < 2\left[\frac{z-3}{2R} \right] $ implies 
 $ v_p\l ( \prod_{i=1}^R f_i(\b t )   \r) < z-3$. 
Note that the contribution of $\b t \in (\Z/ W_z \Z )^{n+1} $ for which there are  $i$ and $p$
such that $v_p( f_i(\b t )   ) \geq m:=  2\left[\frac{z-3}{2R} \right] $ is  
$$ \ll  \sum_{\substack{ i=1,\ldots, R \\ p\leq z}} 
\frac{ \#\{\b t \in (\Z/ W_z \Z )^{n+1}:p^{m } \mid f_i(\b t )  \}}{W_z^{n+1}}=
 \sum_{\substack{ i=1,\ldots, R \\ p\leq z}} 
\frac{ \#\{\b t \in (\Z/ p^{m} \Z )^{n+1}:f_i(\b t )=0  \}}{p^{m(n+1)}}.
$$ By~\eqref{eq:freezers1} this is 
$$ \ll    \sum_{\substack{ i=1,\ldots, R \\ p\leqslant z}} 
\frac{ \#\{\boldsymbol \nu \in (\Z/ p^m \Z )^{R}:\nu_i=0  \}}{p^{{m}R}}\leq R \sum_{p\leqslant z} \frac{1}{p^m}
\ll \frac{z}{2^m}\ll \frac{z}{2^{z/R}} ,$$ which concludes the proof.\end{proof}

  \begin{lemma}
  \label{lem:ozzy}
  For each $\b s \in \{-1,1\}^R$ the limit 
  $$  \gamma(\b s ):=  \lim_{T\to\infty}
\gamma_T(\b s )
  \prod_{p\leq T}\l(1-\frac{1}{p} \r)^{-\frac{R}{2} }  $$ exists
and for any $T>1$  we have 
  $$\gamma_T(\b s )  \prod_{p\leq T}\l(1-\frac{1}{p} \r)^{-\frac{R}{2} }  
=\gamma(\b s )  +O(2^{-T/2} ), $$    where the implied constant depends only on the $f_i$ and $D$.
\end{lemma}Lemma~\ref{lem:ozzy}  will be established
in \S\ref{s:prfoz};
for now we assume its validity and use it to conclude the proof of Theorem~\ref{thm:mainthrm}. 
Injecting Lemmas~\ref{lem:no more tears}-\ref{lem:ozzy} into~\eqref{messiah nostrum}
yields 
$$\frac{N_0(P)}{P^{n+1}}=  
 \l(
\frac{2\gamma_0 }{\sqrt {\pi d (\log P) }}\r)^R
\sum_{\substack{ \b s\in \{1,-1\}^R \\ (D, s_1 \cdots s_R)_\R=1 }} 
\gamma_\infty(\b s ) \gamma(\b s )
+O\l( 
\frac{ (\log P)^{-\frac{R}{2} } }{\min\left\{(\log \log P)^{\frac{c}{2}},
z^{-1} 2^{z/(2R)}
\right\}
}
\r).
$$
By Remark~\ref{rem:immolation}
we obtain 
$$\frac{N_0(P)(\log P)^{R/2}}{P^{n+1}}=  
 \l( \frac{2  }{\sqrt {\pi d   }}\r)^R
\sum_{\substack{ \b s\in \{1,-1\}^R \\ (D, s_1 \cdots s_R)_\R=1 }} 
\gamma_\infty(\b s ) \gamma(\b s )
+O\l( \frac{ 1 }{(\log \log \log P)^A }\r),
$$   for any fixed $A>0$.
Hence, by~\eqref{eq:reni}
we obtain   
$$
N(B) = \gamma\frac{(n+1)^{R/2}}{2 \zeta(n+1)}
 \l(\frac{2 }{\sqrt {\pi d} }\r)^R
\frac{B }{(\log B)^{R/2}}
+O\l( 
\frac{B }{(\log B)^{R/2}}
\frac{ 1 }{(\log \log \log B)^A
}
\r),$$
where
 $$\gamma=\sum_{\substack{ \b s\in \{1,-1\}^R \\ 
(D, s_1 \cdots s_R)_\R=1 }} 
 \gamma(\b s )\cdot \mathrm{vol}
 \l\{\b t \in [-1,1]^{n+1}:   \mathrm{sign}(f_j(\b t ))=s_j \,\, \forall j\r\} $$ 
and  the implied constant depends only on 
$A,D$ and $f_i$.
This proves Theorem~\ref{thm:mainthrm}.

\subsection{Convergence of the leading constant}
\label{s:prfoz}
In this section we 
prove Lemma~\ref{lem:ozzy}.
By the Chinese remainder theorem 
we can might attempt to  write  
the cardinality of $\b t $ in $\gamma_T(\b s )$
as a product of $p$-adic densities. However, the presence of the term 
$\l(\frac{D}{\cdot}\r)$
prevents us from doing this directly as it has no good multiplicative properties for 
the 
primes $p\mid 2D$.
We therefore start by  decomposing the density into bad and     good primes.
Define 
$$
q_1=\prod_{p\mid 2D} p^{T} 
\ \ \textrm{ and } \ \ 
q_2=\prod_{\substack{ p\leq T \\ p\nmid 2D }}p^{T}
.$$
More importantly, we shall use a  product expression
to detect the condition
$(\frac{D}{\cdot})=1$ for all $i$; this will later allow us to write the densities from good primes as a finite sum of Euler products.
\begin{lemma} \label{lem:decomposeunramified}
For $q_1,q_2$ as above 
we have  \beq{def:RITinj}{
\gamma_T(\b s ) 
\prod_{p\leq T}\l(1-\frac{1}{p} \r)^{-\frac{R}{2} }  
 = 2^{-R}\sum_{I\subset \N \cap [1,R]}
\gamma_T^{\mathrm{bad}}(\b s;I ) 
\gamma_{q_2}^{\mathrm{good}}(\b s ;I)  
 ,} where $ \gamma_T^{\mathrm{bad}}(\b s;I ) $ is 
defined through \beq{def:RIT}{ 
\gamma_T^{\mathrm{bad}}(\b s ;I) 
\prod_{p\mid q_1}\l(1-\frac{1}{p} \r)^{R/2} =
q_1^{-n-1}  \sum_{\substack{ \b z \in (\Z/q_1\Z)^{n+1} 
\\ \eqref{eq:cond123}  \textrm{ holds for all }  p\mid q_1 } }
  \prod_{i\in I}    \l(\frac{D}{s_i f_i( \b z)/\gcd(f_i(\b z) ,q_1)}\r)  
} and  for 
$r\in \mathbb N$ we let $$
\gamma_{r}^{\mathrm{good}}(\b s ;I) 
 =
\l( \prod_{p \mid r} \l(1-\frac{1}{p}\r)^{-R/2}\r)
r^{-n-1} 
\sum_{\substack{ \b z \in (\Z/r\Z)^{n+1} 
\\ \eqref{eq:cond123}  \textrm{ holds for all }  p\mid r } }
\prod_{i\in I} \l(\frac{D}{ \gcd(f_i(\b z) ,r)}\r),$$
where \beq{eq:cond123}
{v_p\l (\prod_{i=1}^R f_i(\b z) \r )< T-3,
\ \   \Big(D, \prod_{i=1}^R f_i(\b z )  \Big)_{\Q_p}=1. } \end{lemma}
\begin{proof} Let  $q=q_1q_2$. 
Since $q_1,q_2$ are coprime, 
each $t_i\in \Z/q\Z$ 
is uniquely written as 
$x_iq_2+ y_i q_1$ where $x_i\in \Z/q_1\Z$ 
and $y_i\in \Z/q_2\Z$. 
Then  $f_i(\b t)\equiv q_2^d f_i(\b x) \md{q_1}$ and 
$f_i(\b t)\equiv q_1^d f_i(\b y) \md{q_2}$, 
thus, 
 $\gcd(f_i(\b t) ,q)=
 \gcd(f_i(\b x) ,q_1)
\gcd(f_i(\b y) ,q_2)$. In particular, 
$$
\l(\frac{D}{s_i f_i(\b t)/\gcd(f_i(\b t) ,q)}\r)=
\l(\frac{D}{s_i f_i(\b t)/\gcd(f_i(\b x) ,q_1) }\r)
\l( \frac{D}{\gcd(f_i(\b y) ,q_2)}\r). $$
Recalling that $D$ is assumed square-free
we  note that if $a\equiv b \md{8D}$ then 
$(\frac{D}{a} )=(\frac{D}{b}).$
The upper bound conditions on $v_p$ and $v_2$
in the the counting function in  
$\gamma_T(\b s )$ ensure that 
$$  \frac{ f_i(\b t ) }{\gcd(f_i(\b x ), q_1) }
\equiv \frac{ f_i(q_2\b x ) }
{\gcd(f_i(\b x ), q_1) }\md{8D},$$
hence, 
$(\frac{D}{s_i f_i(\b t)/\gcd(f_i(\b x) ,q_1) })
=(\frac{D}{s_i 
f_i(q_2\b x)/\gcd(f_i(\b x) ,q_1) })$. Thus, 
$$\l(\frac{D}{s_i f_i(\b t)/\gcd(f_i(\b t) ,\prod_{p\leq T} p^T }\r)=1
\iff 
\l(\frac{D}{ \gcd(f_i(\b y) ,q_2)}\r)
 \epsilon_i(\b x ) =1,$$ 
where  $
\epsilon_i(\b x):=
 (\frac{D}{s_i f_i(q_2\b x)/\gcd(f_i(\b x) ,q_1)
} ) 
$. We deduce that the 
succeeding expression takes the value 
$1$ or $0$ according to whether 
$\l(\frac{D}{ \gcd(f_i(\b y) ,q_2)}\r)
 \epsilon_i(\b x ) =1 $ holds for all $i$ or not: 
$$ \prod_{i=1}^R \frac{ 1+
 \epsilon_i(\b x )
\l(\frac{D}{ \gcd(f_i(\b y) ,q_2)}\r)}{2}=
2^{-R}\sum_{I\subset [1,R]}
\prod_{i\in I} \epsilon_i(\b x) 
\l(\frac{D}{ \gcd(f_i(\b y) ,q_2)}\r)
.$$  
The proof  concludes by making the invertible change of variables $q_2 \b x  = \b z$ in $\Z/q_1\Z$.
\end{proof}

\begin{remark} 
This approach to proving convergence differs from the one in \cite[Proposition 3.13]{LRS}, as the proof of Theorem~\ref{thm:mainthrm} at the end 
of~\S\ref{s:activation} requires an explicit error term in Lemma~\ref{lem:ozzy}. To obtain this, we will make use of estimates from the circle method.
\end{remark}

A straightforward argument using the Chinese remainder theorem 
shows that \beq{ghCRT}{ 
\gamma_{q_2}^{\mathrm{good}}(\b s ;I) 
= \prod_{p\mid q_2}  
\gamma_{p^T}^{\mathrm{good}}(\b s ;I) .} 
\begin{lemma} \label{Nicola Porpora} For each    $p\nmid 2D$
the limit   $ 
\gamma_{p^\infty}^{\mathrm{good}}(\b s ;I)  
:= \lim_{T\to \infty } 
\gamma_{p^T}^{\mathrm{good}}(\b s ;I)  
 $ exists. Furthermore,  
\beq{Livietta e Tracollo}{
\gamma_{p^T}^{\mathrm{good}}(\b s ;I)  
= 
\gamma_{p^\infty}^{\mathrm{good}}(\b s ;I)   
+O(p^{-T/2 }),}   \beq{Domenico Natale Sarro}{
\gamma_{p^\infty}^{\mathrm{good}}(\b s ;I)   
    = (1+O(p^{-1-\min\{c',1\}})) 
   \l(1-\frac{ (\frac{D}{p}) }{p} \r)^{-\frac{R}{2}}}
   and 
   \beq{Pastorella In F Major, BWV 590}{
\gamma_{q_2}^{\mathrm{good}}(\b s ;I)  
 = O( 2^{-T/2} )+\prod_{p\mid q_2} 
 \gamma_{p^\infty}^{\mathrm{good}}(\b s ;I)   
,}
where $c'=c'(\b f )$  
and the implied constants only depend on $f_i$ and $D$.
\end{lemma}
\begin{proof}
Since $(D, \prod_{i=1}^R f_i(\b z ) )_p =(\frac{D}{p})^{\sum_{i=1}^R\lambda_i} $, where $\lambda_i=
v_p(  f_i(\b z ) )$, we obtain 
$$ 
\gamma_{p^T}^{\mathrm{good}}(\b s ;I)  
\l(1-\frac{1}{p}\r)^{R/2}
    = \sum_{\substack{ 0\leq 
\sum_{i=1}^R\lambda_i <T    \\ 
1=(\frac{D}{p})^{\sum_{i=1}^R\lambda_i}  
}} 
\frac{\#\{\b z \in (\Z/{p^T}\Z)^{n+1} : v_p(  f_i(\b z) )=\lambda_i \forall i \} }{p^{T(n+1)} }
 \l(\frac{D}{p  }\r)^{\sum_{i\in I}\lambda_i }
.$$By~\eqref{eq:freezers2} the fraction in the right-hand side 
equals   $$ 
p^{-RT}\sum_{\substack{ \boldsymbol \nu \in (\Z/p^T\Z)^{R}\\ v_p(\nu_i)=\lambda_i  \forall i}}
( \sigma_p(\boldsymbol \nu ) +O(p^{-(T+1)(c+1)})  )
=p^{-RT}\sum_{\substack{ \boldsymbol \nu \in (\Z/p^T\Z)^{R}\\ v_p(\nu_i)=\lambda_i  \forall i}}
 \sigma_p(\boldsymbol \nu )  
 +O\l(p^{-(T+1)(c+1)-\sum_{i\geq 1} \lambda_i }
 \r)  ,$$   hence   
 $$ 
\gamma_{p^T}^{\mathrm{good}}(\b s ;I)  
 \l(1-\frac{1}{p}\r)^{R/2}
    = \sum_{\substack{ 0\leq 
\sum_{i=1}^R\lambda_i <T    \\ 
1=(\frac{D}{p})^{\sum_{i=1}^R\lambda_i}  
}}  \sum_{\substack{ \boldsymbol \nu \in (\Z/p^T\Z)^{R}\\ v_p(\nu_i)=\lambda_i \forall i}}
 \frac{ \sigma_p(\boldsymbol \nu )  }{p^{RT}}
 \l(\frac{D}{p  }\r)^{\sum_{i\in I}\lambda_i }
  +O\l(p^{-(T+1)(c+1)  }\r)   .$$
As $T\to \infty $ we see that this converges as the the sum over $\lambda_i$
is absolutely convergent. Indeed, by~\eqref{eq:freezers1}
one has $\sigma_p(\boldsymbol \nu)\ll 1$ uniformly in $p$ and $\boldsymbol \nu$, thus, 
$$
\sum_{\substack{\lambda_1,\ldots, \lambda_R \geq 0  }}
 \sum_{\substack{ \boldsymbol \nu \in (\Z/p^T\Z)^{R}\\ v_p(\nu_i)=\lambda_i \forall i }}
 \frac{ |\sigma_p(\boldsymbol \nu ) | }{p^{RT}}\ll
 \sum_{\substack{\lambda_1,\ldots,\lambda_R \geq 0  }} p^{-\sum_{i\geq 1} \lambda_i }
=(1-1/p)^{-R}<\infty . $$ The 
tail of the sum  is  
$$\leq 
\sum_{\substack{ \lambda \geq T }} 
\sum_{\substack{\lambda_1,\ldots, \lambda_R \geq 0 \\ 
 \sum_{i\geq 1}\lambda_i =\lambda }}
 \sum_{\substack{ \boldsymbol \nu \in (\Z/p^T\Z)^{R}\\ v_p(\nu_i)=\lambda_i  }}
 \frac{ |\sigma_p(\boldsymbol \nu ) | }{p^{RT}}
  \ll   \sum_{\substack{ \lambda \geq T }} 
 \frac{  (\lambda+1)^R }{ p^{\lambda }}
 \ll  p^{-T/2}
 \sum_{\substack{ \lambda \geq 0 }}  
 \frac{  (\lambda+1)^R }{ p^{\lambda }}\ll  p^{-T/2}, $$
 which proves~\eqref{Livietta e Tracollo}.
 We have furthermore proved 
 that  $$
\gamma_{p^\infty}^{\mathrm{good}}(\b s ;I)  
 \l(1-\frac{1}{p}\r)^{R/2}
    = \sum_{\substack{ \lambda_1,\ldots, \lambda_R \geq 0  
    \\ 
1=(\frac{D}{p})^{\sum_{i=1}^R\lambda_i}  
}} \l(\frac{D}{p  }\r)^{\sum_{i\in I}\lambda_i }
 \lim_{T\to \infty}  \sum_{\substack{ \boldsymbol \nu \in (\Z/p^T\Z)^{R}\\ v_p(\nu_i)=\lambda_i \forall i}}
 \frac{ \sigma_p(\boldsymbol \nu )  }{p^{RT}} .$$
 By~\eqref{eq:freezers1} we can see that,
 up to an error term of size $O(p^{-1-c})$,
 the right-hand side is 
\begin{align*}
& \sum_{\substack{ \lambda_1,\ldots, \lambda_R \geq 0  
\\ 1=(\frac{D}{p})^{\sum_{i=1}^R\lambda_i}  
}} \l(\frac{D}{p  }\r)^{\sum_{i\in I}\lambda_i }
 \lim_{T\to \infty}  \sum_{\substack{ \boldsymbol \nu \in (\Z/p^T\Z)^{R}\\ v_p(\nu_i)=\lambda_i \forall i}}
 \frac{ 1 }{p^{RT}} \\
= &  (1-1/p)^R
\sum_{\substack{ \lambda_1,\ldots, \lambda_R \geq 0  
\\ 1=(\frac{D}{p})^{\sum_{i=1}^R\lambda_i}  
}} \l(\frac{D}{p  }\r)^{\sum_{i\in I}\lambda_i }
p^{-\sum_{i\geq 1 } \lambda_i } 
= O(p^{-2}) + (1-1/p)^R  \c B
,\end{align*} as  the cases with $\sum_{i=1}^R \lambda_i \geq 2 $
trivially contribute $O(p^{-2})$ and where 
$$ \c B= \sum_{\substack{ 0\leq \sum_{i=1}^R
\lambda_i \leq 1  
\\ 1=(\frac{D}{p})^{\sum_{i=1}^R\lambda_i}  
}} \l(\frac{D}{p  }\r)^{\sum_{i\in I}\lambda_i }
p^{-\sum_{i\geq 1 } \lambda_i } .
$$   The terms where each $\lambda_i$ vanishes 
clearly   $1$ to $\c B$.
For the remaining terms
there exists $j\in [1,R]$ such that $\lambda_j=1$ and   $\lambda_i=0$ for $i\neq j$. 
If $(\frac{D}{p})=1$ then the contribution to $\c B$ is 
$\frac{R}{p}$. If $(\frac{D}{p})=-1$ then there are no  terms to consider as the conditions 
$1=(\frac{D}{p})^{\sum_{i\geq 1 } \lambda_i } $ 
and $\sum_{i\geq 1 } \lambda_i =1 $  
are incompatible. This yields 
 $$\gamma_{p^\infty}^{\mathrm{good}}(\b s ;I)   
    = O(p^{-1-\min\{c',1\}})+
    \l(1-\frac{1}{p}\r)^{R/2}
    \begin{cases}
1+\frac{R}{p},  &  (\frac{D}{p})=1, \\
1, & (\frac{D}{p})=-1,
\end{cases}
$$ which is sufficient for~\eqref{Domenico Natale Sarro}.  
 Lastly,  by~\eqref{ghCRT},\eqref{Livietta e Tracollo} 
   and~\eqref{Domenico Natale Sarro} we infer that 
$$ 
\gamma_{q_2}^{\mathrm{good}}(\b s ;I)   
- \prod_{p\mid q_2} 
\gamma_{p^\infty}^{\mathrm{good}}(\b s ;I)  
\ll 2^{-T/2},$$ which 
proves~\eqref{Pastorella In F Major, BWV 590}
\end{proof}

 \begin{lemma} \label{Nicola Porpora234} The limit  
 $\gamma^{\mathrm{bad}}(\b s;I ) 
 := \lim_{T\to \infty } 
\gamma_T^{\mathrm{bad}}(\b s;I )  
 $ exists. Furthermore,  
\beq{Livietta e Tracollo234}{
\gamma^{\mathrm{bad}}(\b s;I ) 
=  \gamma_T^{\mathrm{bad}}(\b s;I ) 
 +O(2^{-T }),}   
where the implied constants only depend on $f_i$ and $D$.
\end{lemma}
\begin{proof} Let $\epsilon_2=4$ and let $\epsilon_p=1$ for an odd prime $p\mid D$.
For each $\b z $ in the definition of 
$\gamma_T^{\mathrm{bad}}(\b s;I )  
$ and each   $p\mid q_1$  we let 
$\alpha_p=v_p (\prod_{i=1}^R f_i(\b z) )$. By Remark~\ref{remIHP} 
we have $  (D, \prod_{i=1}^R f_i(\b z )   )_p=   (D, \prod_{i=1}^R f_i(\b x )  )_p$ for each $p\mid q_1$,
where   $\b x \in (\Z/p^{\alpha_p+\epsilon_p}\Z)^{n+1}$
is the reduction of $\b z \md{ p^{\alpha_p+\epsilon_p} }$. 
Since $(\frac{D}{a})=(\frac{D}{b})$ for all integers $a,b$ with $a\equiv b \md {8D }$ 
we also have 
$$\l(\frac{D}{s_i f_i(\b z)/\gcd(f_i(\b z) ,q_1) }\r)  =\l(\frac{D}{s_i f_i(\b x)/\gcd(f_i(\b x) ,q_1) }\r)  $$
because the bounds on $v_p$ in~\eqref{eq:cond123} ensure that $v_p(  f_i(\b z ))=
v_p(  f_i(\b x ))$ for all $i$ and therefore we can   get  
$  f_i(\b z) p^{-v_p(f_i(\b z))} \equiv  f_i(\b x) p^{-v_p(f_i(\b x))} \md{p^{\epsilon_p}}$,
thus, multiplifying by $\prod_{\ell \mid 2D, \ell \neq p }\ell^{-v_\ell(f_i(\b z) )}$ we get 
$$ \frac{ f_i(\b z)}{\gcd (f_i(\b z) ,q_0) } \equiv \frac{ f_i(\b x)}
{\gcd (f_i(\b x) ,q_1) }\md{p^{\epsilon_p}},$$ where 
 $q_0=\prod_{p\mid 2D} p^{\alpha_p +\epsilon_p}$. 
 This means that the right-hand side of~\eqref{def:RIT} can be written as 
 $$ 
\sum_{\substack{ 0\leq  \alpha_p \leq T-\epsilon_p \forall p\mid 2D } }
\sum_{\substack{ \b x \in (\Z/q_0\Z)^{n+1 } \\ \eqref{eq:cond123}'  \textrm{ holds for all }  p\mid q_0 } }
 \hspace{-0.5cm}
 \frac{\#\{ \b z \in (\Z/q_1\Z)^{n+1} :q_0 \mid \b z - \b x    \}}{q_1^{n+1} }
  \prod_{i\in I}    \l(\frac{D}{s_i f_i( \b z)/\gcd(f_i(\b z) ,q_1)}\r)  , $$ where 
   the 
  notation~$\eqref{eq:cond123}'$
 means that~\eqref{eq:cond123} holds with $\b x$ in place of $\b z$.
 Since the cardinality of $\b z $ in the last formula is 
 $  (q_1/q_0)^{n+1} $ we obtain 
\begin{align*}
&\gamma_T^{\mathrm{bad}}(\b s;I ) 
 \prod_{p\mid q_1}\l(1-\frac{1}{p} \r)^{\frac{R}{2} }
 \\ 
 =&  
 \sum_{\substack{ 0\leq  \alpha_p \leq T-\epsilon_p \forall p\mid 2D } }
 \frac{1}{\prod_{p\mid 2D} p^{(n+1)(\alpha_p+ \epsilon_p)}}
\sum_{\substack{ \b x \in (\Z/q_0\Z)^{n+1 } \\ \eqref{eq:cond123}'  \textrm{ holds for all }  p\mid q_0 } }
  \prod_{i\in I}    \l(\frac{D}{s_i f_i( \b z)/\gcd(f_i(\b z) ,q_1)}\r)  . \end{align*}
  This sum converges absolutely as $T\to \infty$, 
  thus,   $\lim_{T\to\infty} 
\gamma_T^{\mathrm{bad}}(\b s;I ) 
  $ exists. The   error term is
  $$ \ll  
  \sum_{ p\mid 2D } \sum_{\alpha_p \geq T+1-\epsilon_p}
  \frac{1}{  p^{\alpha_p+ \epsilon_p}} \ll 2^{-T}.  $$
 \end{proof}

The proof of Lemma~\ref{lem:ozzy} follows by 
injecting~\eqref{Pastorella In F Major, BWV 590} 
and~\eqref{Livietta e Tracollo234} into~\eqref{def:RITinj}.
Furthermore, we get 
\beq{eq:eulerproducts!}{
\gamma(\b s)= 2^{-R}\sum_{I\subset \N \cap [1,R]}
\gamma^{\mathrm{bad}}(\b s;I ) 
 \mathfrak S(\b s ;I)  ,}
where $ \mathfrak S$ is the Euler product given by 
\beq{eq:you are revolting, trully disgusting}{
\mathfrak S(\b s ;I) :=
\prod_{\substack{ p=3\\ p\nmid 2D}}^\infty 
\gamma_{p^\infty}^{\mathrm{good}}(\b s ;I).}

\section{The leading constant}
\label{s:Julian the Apostate}%
In the previous sections we have obtained an asymptotic formula for the number of $\mathbf t \in \mathbb P^n(k)$ of bounded height for which the conic
\begin{equation}\label{eq:conic}
    x^2-Dy^2 = z^2 \prod_{i=1}^R f_i(\mathbf t)
\end{equation}
in $\mathbb P^2$ has a point. To interpret this result geometrically we will use the projective conic bundle $f \colon X \to \mathbb P^n$ whose fibre over $\mathbf t \in \mathbb P^n$ is given by \eqref{eq:conic}. To construct this total space $X$ we could define it on the chart $t_j\ne 0$ as a subscheme of $\mathbb A^n \times \mathbb P^2$. We can glue these charts together as $dR$ is even as one does for example for classical Ch\^atelet surfaces, and is explained in detail in \cite[\textsection 1.1]{SofosVisse}, which produces a scheme $X$ which is smooth and proper over the base field $k$.
A huge component of the geometric interpretation relates to the non-split fibres as $X$ is smooth and proper, so let us note that the fibres of $f$ are split away from the non-empty open subscheme
$$
U := \P^n \setminus 
\bigcup_{i=1}^R \{f_i(\b t)=0\}.
$$
We define  $V:= f^{-1}(U) \subseteq X$ as the preimage of $U$. One also easily verifies that $f$ is smooth away from the intersections $\{f_i(\textbf t) = 0 = f_j(\textbf t)\}$ for $i \ne j$.

We can now rephrase the asymptotic result in Theorem~\ref{thm:mainthrm} for \( B \to \infty \) as
\begin{equation}\label{eq:O(n+1) counting result}
\frac{  \#\{ \textbf t \in \P^n(\Z) \colon H(\textbf t)
  \leq B, f^{-1}(\textbf t)(\Q) \ne \emptyset\} }
  {B(\log B)^{-R/2} }
  \underset{B \to +\infty}{\longrightarrow}
\frac{1}{2 \zeta(n+1)} 
 \l(
\frac{2\sqrt{n+1} }{\sqrt {\pi d} }\r)^R
\hspace{-0.3cm}
\sum_{\substack{ \b s\in \{1,-1\}^R
\\ (D, s_1 \cdots s_R)_\R=1 }}  \hspace{-0.2cm}
\gamma_{\infty}(\b s )
\gamma (\b s ) \end{equation} 
for 
$\gamma (\b s )$ as in  Theorem~\ref{thm:mainthrm}
and
$\gamma_{\infty}(\b s )= \mu_\infty \l(
\b t \in [-1,1]^{n+1}:
  s_j f_j(\b t ) >0\r)$, where $\mu_\infty$ is the standard Lebesgue measure. For smooth conics the Hasse principle holds, and so we could equivalently consider the everywhere locally soluble fibres.
An asymptotic for this counting problem under an anticanonical adelic height has been conjectured in \cite[Conjecture 3.8]{LRS}
\[
\frac{\#\{ \textbf t \in \P^n(\Z) \colon H(\textbf t)
  \leq B, f^{-1}(\textbf t)(\Q) \ne \emptyset\} }
  {B(\log B)^{-\Delta(f)} }
  \to   c_{f, \text{pred}}  
\]
in terms of the fibration $\pi \colon X \to \mathbb P^n$ with a smooth and proper total space, since $\rho(\P^n) := \rank \Pic \P^n=1$. To be precise, the predicted constant equals
\[
c_{f, \text{pred}} = \frac{\alpha^*(\P^n) \#\left[ \Brsub(\P^n,f)/\Br \Q\right] \cdot \tau_f\left( \big[\prod_v fX(\Q_v)\big]^{\Brsub} \right)}{\Gamma(\P^n,f)} \prod_{D \in (\P^n)^{(1)}} \eta(D)^{1-\delta_D(f)}.
\]

\begin{proposition}\label{prop:comparing constants}
    The predicted order of growth and leading constant $c_{f,\text{pred}}$ agrees with the order and constant in \eqref{eq:O(n+1) counting result}.   
\end{proposition}

We will recall the notation in the predicted constant as we compute each factor separately. The only part that solely depends on the base is the classical $\alpha^*(\P^n) = \alpha(\P^n)/(\rho(\P^n)-1)! =\frac 1{n+1}$.

\subsection{The numerical invariants of the fibration}

The fibration $f$ is smooth away from the zero locus of the $f_i$. Hence, the only points of codimension $1$ with a non-split fibre are the generic points of the prime divisors $D_i=\{f_i(\b t) = 0\}$.

\begin{lemma}\label{lem:numerical invariants}
    We have $\delta_{D_i}(f) = \frac12$, $\Delta(f) = \frac R2$, $\eta(D_i)=\frac{n+1}d$, and $\Gamma(\P^n,f) = \sqrt \pi^R$.
\end{lemma}

\begin{proof}
Since the $f_i$ are as in  Corollary~\ref{cor:birch}, the divisors $D_i$ are geometrically irreducible as proven in \cite[\textsection 1.2]{SofosVisse}. The fibre $\{x^2-Dy^2=0, f_i(\b t)=0\}$ over $D_i$ consists of two geometrical components interchanged by the non-trivial element of the relevant Galois group, hence $\delta_{D_i}(f)=\frac12$. We conclude
    \[
    \Delta(f) = \sum_{D} \left( 1- \delta_D(f)\right) = \frac R2.
    \]
    As the anticanonical divisor of the base $\P^n$ is of degree $n+1$ and $\deg f_i=d$ we find $\eta(D_i)=\frac{n+1}d$ for all $i$ by definition \cite[(3.3)]{LRS}.

    For $\Gamma(\P^n,f)$ the definition in 
    \cite[Conjecture 3.8]{LRS} shows that it equals
    $\Gamma(\tfrac12)^R={\sqrt \pi}^R$.    
\end{proof}

As a corollary we obtain that the product of Fujita invariants equals
\[
\prod_{D \in (\P^n)^{(1)}} \eta(D)^{1-\delta_D(f)} =\sqrt{\frac{n+1}d}^R. 
\]

\subsection{The subordinate Brauer group}

We will see that the reciprocity condition involving the quadratic symbol $(\frac{D}{\cdot})$
in Lemma~\ref{lem:no more tears}
comes from the following ramified elements of the Brauer group. 
This distinction is reminiscent of the case of Ch\^atelet type surfaces $x^2-Dy^2=f(x)$ with $f$ of even degree (see \cite{CTS,CTSSD} and \cite[Chapter 7]{Sko}). If $f$ is irreducible or the product of irreducible factors of odd degree, then the unramified Brauer group  is constant, and the Hasse principle and weak approximation are conjectured to hold, and known to hold when $\deg(f)=4$. If however $f$ factors into two irreducible factors of even degree, then the unramified Brauer group is isomorphic to $\Z /2\Z$ and one obtains a Brauer--Manin obstruction and counter-examples to weak approximation and to the Hasse principle.

\begin{proposition}\label{prop:subordinate brauer group}
    The subordinate Brauer group modulo constants is generated by $(D,g(\b t)) \in \Br \kappa(\P^n)$ for $g \in \mathcal G$ where
    \[
    \mathcal G =
    \begin{cases}
        \{f_1,\ldots, f_R\}, & \text{ if $\deg f_i$ is even for all $i$};\\
        \{f_1f_2,\ldots, f_1f_R\}, & \text{ if $\deg f_i$ is odd for all $i$}.
    \end{cases}
    \]
    Hence $\#\left[ \Brsub(\P^n,f)/\Br \Q\right]$ equals $2^R$ in the even degree case, and $2^{R-1}$ in the odd degree case.
\end{proposition}

\begin{proof}
    By definition of the subordinate Brauer group we are looking at the elements $\mathcal A \in \Br \kappa(\P^n)$ which are unramified away from the geometrically irreducible divisors $D_i := \{f_i =0\}$, and whose residue along each $D_i$ is either $1,D \in \kappa(D_i)^\times/\kappa(D_i)^{\times,2}$. By Theorem~3.6.1(ii) in \cite{brauerbook} the class of the element $\mathcal A$ modulo constants is uniquely determined by these constraints if it exists, this proves that there are at most $2^R$ classes of the subordinate Brauer group modulo constants.

    In the case that $d$ is even we see that the elements $(D,\prod_{i \in I} f_i)$ for $I\subseteq \{1,\dots,R\}$ are the $2^R$ elements we were looking for.

    In the case that $d$ is odd we have the $2^{R-1}$ elements $(D,\prod_{i \in I} f_i)$ with $\# I$ even. To show that other combinations of residues do not occur, it suffices to show that there is no $\mathcal A \in \Br \kappa(\P^n)$ which is unramified away from $D_1$ and has residue $D$ along $D_1$. Let us consider a line $L \subseteq \P^n$, which is not contained in $D_1$. We will study $\left.\mathcal A\right|_L \in \Br\kappa(L)$ and proceed by contradiction.

    As $\mathcal A$ is unramified away from $D_1$, the restriction $\left.\mathcal A\right|_L$ is unramified away from $D_1 \cap L$. This intersection is supported on finitely many closed points $P_i \in L$ of degree $d_i$ and multiplicity $m_i$. As $D_1$ has degree $d$ we find $\sum_i m_id_i=d$. The residue at $P_i$ equals $D^{m_i} \in \kappa(P_i)^\times/\kappa(P_i)^{\times,2}$ by \cite[Theorem 3.7.5]{brauerbook} which applies as $D_1$ is smooth. By Faddeev's exact sequence \cite[Theorem 1.5.2]{brauerbook} we find that, multiplicatively written and modulo squares, $1 = \prod_i \mbox{cores}_{\kappa(P_i)/\mathbb{Q}}(D^{m_i}) = \prod_i D^{d_im_i} = D^d$. This is a contradiction as $d$ is odd, and $D$ is not a square.
\end{proof}

One can also prove that the vertical Brauer group of $f$ is constant, and the natural inclusion $\Brvert(f) \subseteq f^* \Brsub(\P^n,f)$ is not an equality. This shows that there is no smooth proper model of $X$ over $\P^n$ which is flat over the base, as in \cite[Lemma 2.10]{LRS}.

\subsection{The Tamagawa number}

We   compute the Tamagawa measure $\tau_f$ of $\big[\prod_v fX(\Q_v)\big]^{\Brsub}$. To do so, recall that the subordinate Tamagawa number     in \cite[(3.10)]{LRS} as the limit 
\[
\tau_f\left(\big[\prod_v fX(\Q_v)\big]^{\Brsub}\right) := \lim_{\stackrel{T \to \infty}{S=\{p\leq T\} \cup \{\infty\}}} \l[ \prod_{v \in S} \lambda^{-1}_v \tau_v \r] \l(\l[ \prod_{v \in S} f V(\Q_v)\r]^{\Brsub} \r),
\]
since $\lim_{s\to 1} \l[(s-1)\zeta(s)\r]^{1-\frac R2} = 1$. Some remarks are in order. In \cite{LRS} this limit is written as a limit over the set $S$, but it is crucial that $S$
has the shape 
$\{p\leq T\}$ with $T$ increasing,
since the infinite product is only conditionally convergent in many examples, as it is in this paper.  
Secondly, the convergence of this infinite product in our case is not immediately clear. Rather than relying on 
\cite[Lemma 3.13]{LRS}, we establish convergence using 
the circle method estimates in Lemma~\ref{lem:ozzy}.

\begin{lemma}\label{lem:Tamagawa part for finite S}
    Let $S$ be a finite set of  
    places which contains $\infty$, $2$ and all primes dividing $D$. Let $q$ be an integer divisible by all primes in $S$.
    
    We have 
    \begin{align*}
     & \#\left[ \Brsub(\P^n,f)/\Br \Q\right]\cdot \l[ \prod_{v \in S} \lambda^{-1}_v \tau_v \r] \l(\l[ \prod_{v \in S} f V(\Q_v)\r]^{\Brsub} \r)  = &\\
     & \qquad \qquad \qquad \qquad \l(\frac{n+1}2 2^R \prod_{p \in S \setminus \infty}  ( 1- p^{-n-1} ) \r) \sum_{\substack{s_1,\ldots, s_R \in \{-1,1\}\\ (D,s_1\ldots s_R)_\R=1}} \gamma_\infty(\b s) t_S(\b s), 
    \end{align*}
    where $t_S(\b s) =  \left[ \prod_{p \in S}  \left( 1- \frac1p\right)^{-\frac R2} \mu_p \right]  \l \{\b t \in  \Z_q^{n+1} \colon \ \begin{array}{l} (D,  f_1(\b t ) \cdots f_R(\b t ) )_{\Q_p}=1 \ \forall p \in S\setminus \infty,  \\ \left(\frac D{s_if_i(\b t)/\prod_{p \in S\setminus \infty} p^{v_p(f_i(\b t))}}\right) = 1 \  \forall i\end{array}\right\} $.   
\end{lemma}

\begin{proof} We recall from Proposition~\ref{prop:subordinate brauer group} that the subordinate Brauer group is generated by $(D,g_i)$ for the elements
    \[
    \begin{cases}
    g_i = f_i,\quad 1 \leq i \leq R & \text{ if } 
    d  
    \text{ is even},\\
    g_i = f_1f_i,\quad 2 \leq i \leq R & \text{ if } 
    d 
    \text{ is odd}.
    \end{cases}
    \] 
    With $g_i$ as in Proposition~\ref{prop:subordinate brauer group}
    we have by definition 
        \begin{align*}
    \big[\prod_{v \in S} f V(\Q_v)\big]^{\Brsub} & = \l \{(\b t_v)_v \in  \prod_{v\in S} \P^n(\Z_v) \colon \quad \begin{array}{l}
(D,  f_1(\b t_v ) \cdots f_R(\b t_v ) )_{\Q_v}=1 \ \forall v \in S,  \\
\prod_{v\in S} (D,g_i(\b t_v)_{\Q_v} = 1 \ \forall i
\end{array}\right \}. 
\end{align*}
Recall that the Tamagawa measure $\prod_v \lambda_v^{-1} \tau_v$ is constructed, see \cite[\textsection 3.6]{LRS}, using the convergence factors $\lambda_p = \left(1-\frac 1p\right)^{\frac R2 + 1}$. By \cite[Proposition 4.1]{LRS} we can compute the Tamagawa measure of this set using the measure
\begin{align*}
& \frac{n+1}{2} \mu_\infty \prod_{p \in S \setminus \infty} \left[ (1+p^{-1} + \cdots + p^{-n}) \left(1-\frac 1p\right)^{-\frac R2 - 1} \mu_p\right] = &\\
& \qquad \qquad \qquad \qquad \qquad  \left(\frac{n+1}2 \prod_{p \in S \setminus \infty}  ( 1- p^{-n-1} ) \right) \cdot \mu_\infty \times \prod_{p \in S \setminus \infty} \left[\left(1-\frac 1p\right)^{-\frac R2} \mu_p\right]
\end{align*}
on the corresponding affine cone, where the 
$\mu_p$ and
$\mu_\infty$ 
are the appropriate Haar measures on $\mathbb Z^{n+1}_p$ and $\R^{n+1}$. We conclude that up to the factor converging absolutely to $\frac{n+1}{2\zeta(n+1)}$, we can compute the Tamagawa measure as
$$
\mu_\infty \times  
\prod_{p \in S\setminus \infty}
\left(1-\frac1p \right)^{-\frac R2} 
\mu_p
\l \{(\b t_v)_v \in  \R^{n+1} \times 
\prod_{p \in S\setminus \infty}\Z^{n+1}_p \colon 
\begin{array}{l}
(D,  f_1(\b t_v ) \cdots f_R(\b t_v ) )_{\Q_v}=1 \ \forall v \in S,  \\
\prod_{v\in S} (D,g_i(\b t_v)_{\Q_v} = 1 \ \forall i
\end{array}\right \}. 
$$
We will rewrite the set in this equation 
using $\prod_{p\in S\setminus \infty} \Z_p = \Z_q$ as $q$ is an integer divisible by all finite places in $S$. We find after taking out the archimedean place and partitioning by the signs $s_i$ of the $f_i$ a disjoint union
\[
\l \{(\b t_v)_v \in  \R^{n+1} \times \prod_{p \in S\setminus \infty}\Z^{n+1}_p \colon \ \begin{array}{l}
(D,  f_1(\b t_v ) \cdots f_R(\b t_v ) )_{\Q_v}=1 \ \forall v \in S,  \\
\prod_{v\in S} (D,g_i(\b t_v)_{\Q_v} = 1 \ \forall i
\end{array}\right \}=
\]
\begin{align*}
& \coprod_{s_i \in \{-1,1\}} \Bigg[ \l\{\b t_\infty \in \R^{n+1} \colon \  \begin{array}{l}
(D,  f_1(\b t_\infty ) \cdots f_R(\b t_\infty ) )_{\R}=1,  \\
s_i f_i(\b t_\infty)>0\ \forall i
\end{array}
\r\} \\
    & \qquad \qquad \qquad \qquad \qquad \times \l \{(\b t_p)_p \in  \Z_q^{n+1} \colon \  \begin{array}{l}
(D,  f_1(\b t_p ) \cdots f_R(\b t_p ) )_{\Q_p}=1 \ \forall p \in S\setminus \infty,  \\
\prod_{v\in S} (D,s'_i g_i(\b t_v)_{\Q_v} = 1 \ \forall i
\end{array}\right\} \Bigg] =\\
& \coprod_{\substack{s_i \in \{-1,1\}\\ (D,\prod s_i)=1}} 
    \Bigg[ \l\{\b t_\infty \in \R^{n+1} \colon s_i f_i(\b t_\infty)>0 \ \forall i
\r\}\\
& \qquad \qquad \qquad \qquad \qquad \times \l \{(\b t_p)_p \in  \Z^{n+1}_q \colon \ \begin{array}{l}
(D,  f_1(\b t_p ) \cdots f_R(\b t_p ) )_{\Q_p}=1 \ \forall p \in S\setminus \infty,  \\
(D,s'_i)_\R \prod_{p\in S\setminus \infty} (D,g_i(\b t_p))_{\Q_p} = 1 \ \forall i
\end{array}\right\} \Bigg],
    \end{align*}
where $s'_i=s_i$ if the $d$ is even, and $s'_i = s_1 s_i$ if $d$ is odd. We will now rewrite the last conditions in terms of Kronecker symbols. To that end introduce $g_i(\b t) = u_i P_i$ with $u_i \in \Z_q^\times$ and $P_i=\prod_{p\in S\setminus \infty} p^{e_p}$, then
\[
1 = (D,s'_i)_\R \prod_{p\in S\setminus \infty} (D,g_i(\b t))_{\Q_p} = \left( \frac D{s'_i}\right) \prod_{p\in S\setminus \infty} (D,u_i)_{\Q_p}(D,P_i)_{\Q_p}.
\]
Note that $\prod_{p \in S \setminus \infty} (D,P_i)_{\Q_p} =(D,P_i)_\R \prod_{p\notin S} (D,P_i)_{\Q_p} = 1$ as $P_i > 0$ and all primes dividing $D$ or~$P_i$ lie in $S$. Let us approximate $u_i$ by an integer, so we can write
\begin{align*}
1 & = \left( \frac D{s'_i}\right) \prod_{p\in S\setminus \infty} (D,u_i)_{\Q_p} = \left( \frac D{s'_i}\right) (D,u_i)_\R \prod_{p \notin S} (D,u_i)_{\Q_p}\\
& = \left( \frac D{s'_i}\right) (D,u_i)_\R \prod_{p \notin S} \left( \frac Dp \right)^{v_p(u_i)} = \left( \frac D{s'_i}\right)\left(\frac D{u_i}\right) = \left(\frac D{s'_iu_i}\right).
\end{align*}
Hence we find for a fixed vector $\b s$ of signs
\[
\l\{\b t_\infty \in \R^{n+1} \colon s_i f_i(\b t_\infty)>0 \ \forall i
\r\} \times \l \{(\b t_p)_p \in  \Z^{n+1}_q \colon \ \begin{array}{l}
(D,  f_1(\b t_p ) \cdots f_R(\b t_p ) )_{\Q_p}=1 \ \forall p \in S\setminus \infty,  \\
\left( \frac{D}{s'_ig_i(\b t)/\prod_{p \in S \setminus \infty} p^{v_p(g_i(\b t))}} \right) = 1 \ \forall i
\end{array}\right\}.
\]
If $d$ is even, and hence $g_i=f_i$ and $s'_i = s_i$ we recover that the measure equals $\gamma_\infty(\b s)t_S(\b s)$, which proves the statement in this case as the order of the subordinate Brauer group modulo constants is $2^R$.

If however, $d$ is odd, we find
\begin{align*}
& \l\{\b t_\infty \in \R^{n+1} \colon s_i f_i(\b t_\infty)>0 \ \forall i
\r\}\\
& \quad \quad \quad \quad \quad \times \l \{(\b t_p)_p \in  \Z^{n+1}_q \colon \ \begin{array}{l}
(D,  f_1(\b t_p ) \cdots f_R(\b t_p ) )_{\Q_p}=1 \ \forall p \in S\setminus \infty,  \\
\left( \frac{D}{s_if_i(\b t)/\prod_{p \in S \setminus \infty} p^{v_p(f_i(\b t))}} \right) = \left( \frac{D}{s_1f_1(\b t)/\prod_{p \in S \setminus \infty} p^{v_p(f_1(\b t))}} \right) \ \forall i
\end{array}\right\}\\
= & 
\l\{\b t_\infty \in \R^{n+1} \colon s_i f_i(\b t_\infty)>0 \ \forall i
\r\}\\
& \quad \quad \quad \quad \quad \times \coprod_{\psi \in \{-1,1\}} \l \{(\b t_p)_p \in  \Z^{n+1}_q \colon \ \begin{array}{l}
(D,  f_1(\b t_p ) \cdots f_R(\b t_p ) )_{\Q_p}=1 \ \forall p \in S\setminus \infty,  \\
\left( \frac{D}{s_if_i(\b t)/\prod_{p \in S \setminus \infty} p^{v_p(f_i(\b t))}} \right) = \psi \ \forall i
\end{array}\right\}.
\end{align*}
Let $\lambda \in \left(\Z/D\Z\right)^\times$ for which $\left(\frac{D}\lambda\right)=-1$, which exists as $S$ contains all places dividing $D$. Then multiplication on $\Z_q^{n+1}$ by $\lambda$ gives a measure preserving bijection between the two sets with $\psi=1$ and $\psi=-1$, as $d$ is odd. Hence the measure of this set for a single $\b s$ equals $2 \gamma_\infty(\b s) t_S(\b s)$. We now conclude as the subordinate Brauer group modulo constants in this case has order $2^{R-1}$ .
\end{proof}

Recall the definition of  $\gamma(\b  s)$ in 
Theorem~\ref{thm:mainthrm}. To compare $t_S(\b s)$ to the term of the limit defining 
$\gamma(\b s)$ we shall use Lemma~\ref{lem:mitsotakigamiesai2}.

\begin{proposition}\label{prop:limit of finite part of Tamagaw is kappa}
    Take $q=\prod_{p\leq T} p^T$. Then
    \[
\gamma(\b s) = \lim_{T \to \infty} \left[ \prod_{p \leq T} \left(1-\frac1p\right)^{-\frac R2} \mu_p \right] \l \{\b t \in  \Z^{n+1}_q \colon \ \begin{array}{l} (D,  f_1(\b t ) \cdots f_R(\b t ) )_{\Q_p}=1 \ \forall p \leq T,  \\ \left(\frac D{s_if_i(\b t)/\prod_{p \leq T} p^{v_p(f_i(\b t))}}\right) = 1 \  \forall i\end{array}\right\}.
    \]
\end{proposition}
\begin{proof}
The terms of the limit agree with
    \[
    \left[\prod_{p \leq T} \left(1-\frac1p\right)^{-\frac R2} 
    \mu_p \right]\l\{\b t \in  \Z^{n+1}_q \colon \ \begin{array}{l} v_p(\prod_i f_i(\b t)) < T \ \forall 3 \leqslant p \leqslant T\\[2mm]
    v_2(\prod_i f_i(\b t)) < T- 3\\
    (D,  f_1(\b t ) \cdots f_R(\b t ) )_{\Q_p}=1 \ \forall p \in S\setminus \infty,  \\ \left(\frac D{s_if_i(\b t)/\prod_{p \in S\setminus \infty} p^{v_p(f_i(\b t))}}\right) = 1 \  \forall i\end{array}\right\}
    \]
    up to 
    $$\ll \sum_{p\leq T}\prod_{p \leq T} 
    \l(1-\frac1p\r)^{-R/2} 
     p^{-T/R} 
 \ll ( \log T)^{R/2} 
T2^{-T/R} \underset{T \to +\infty}{\longrightarrow} 0 
.$$  
To prove the bound, let 
$a_2=3$ and $a_p=0$ for all other $p$
and  note that if  
$v_p(\prod_i f_i(\b t)) \geq  T-a_p $   for some $p \leq T$ 
then there exists  $p\leq T$
and $1\leq i \leq R $ such that 
$v_p(f_i(\b t )) \geq (T-a_p)/R$.
By Lemma~\ref{lem:mitsotakigamiesai2} this probability is 
$\ll p^{-(T-a_p)/R}\leq 2 p^{-T/R}$.

    The conditions in the latter expression only depend on $\b t \in \Z_q^{n+1}$ modulo $q$, hence we can replace the Haar measure by $q^{-n-1}$ times the counting measure on $(\Z/q\Z)^{n+1}$. This recovers the definition of $\gamma(\b s)$, thereby proving the claim. 
\end{proof}

\subsection{Comparing the leading constant with the prediction}

We are now ready to put all parts of the prediction together.

\begin{proof}[Proof of Proposition~\ref{prop:comparing constants}]
    From Lemma~\ref{lem:numerical invariants}
    and Lemma~\ref{lem:Tamagawa part for finite S} we see that order of growth $B/(\log B)^{\frac R2}$ is as predicted, and that
    \begin{align*}
    c_{f,\text{pred}} = & \frac 1{n+1}\frac1{\sqrt \pi^R} \l(\frac{n+1}d \r)^{\frac R2} \#\left[ \Brsub(\P^n,f)/\Br \Q\right] \cdot \lim_S \tau_f\left( \big[\prod_v f(X(\Q_v))\big]^{\mathcal B, S}\right)\\
            = & \frac 1{n+1} \l(\frac{n+1}{\pi d} \r)^{\frac R2} \frac{n+1}2 \frac{2^R}{\zeta(n+1)} \lim_{\substack{T \to \infty\\ S=\{p\leq T\}}} \left[ \sum_{\substack{s_1,\ldots, s_R \in \{-1,1\}\\ (D,s_1\ldots s_R)=1}} 
            \gamma_\infty(\b s) t_S(\b s) \right]\\
    = & \frac1{2} \l(\frac{n+1}{\pi d} \r)^{\frac R2} \frac{2^R}{\zeta(n+1)} \sum_{\substack{s_1,\ldots, s_R \in \{-1,1\}\\ (D,s_1\ldots s_R)=1}} \gamma_\infty(\b s) \lim_{\substack{T \to \infty\\S=\{p\leq T\}}} t_S(\b s).
    \end{align*} 
We know that  $t_S(\b s)$ converges to $\gamma(\b s)$ if $S$ increases with $T$ by Lemma~\ref{prop:limit of finite part of Tamagaw is kappa}. Hence we arrive at
    \[
    c_{f,\text{pred}} = \l(\frac{n+1}{\pi d} \r)^{\frac R2}\frac{2^R}{2\zeta(n+1)} \sum_{\substack{s_1,\ldots, s_R \in \{-1,1\}\\ (D,s_1\ldots s_R)_\R=1}} \gamma_\infty(\b s) \gamma(s),
    \]
    which agrees with the leading constant in \eqref{eq:O(n+1) counting result}.
\end{proof}

\end{document}